\documentclass[11pt]{amsart}
 \usepackage[normalem]{ulem}
\usepackage{mathtools}
\mathtoolsset{showonlyrefs,showmanualtags} 

\usepackage{amssymb, amsmath, amsthm, esint}
\usepackage{mathrsfs}
\usepackage{bbm,dsfont}
\usepackage{hyperref}
\usepackage{mathtools}
\usepackage{fullpage}
\usepackage{color}
\usepackage{tikz}
\usepackage{comment} 
\usepackage{graphicx} 
\usepackage{footnote}
\usepackage{bm}
\usepackage[backrefs]{amsrefs}

 \usepackage[top=1.5in, bottom=1.5in, left=1.25in, right=1.25in]	{geometry}
\usepackage[all]{xy}
\usepackage{amsmath}
\usepackage{amssymb}
\usepackage{amsthm}
\usepackage{amscd}

\newcommand{\NN}{\mathbb  N}
\newcommand{\Z}{\mathbb  Z}

\numberwithin{equation}{section}

\newtheorem{theorem}[equation]{Theorem}
\newtheorem{definition}[equation]{Definition}
\newtheorem{proposition}[equation]{Proposition}

\newtheorem{lemma}[equation]{Lemma}

\newtheorem{remark}[equation]{Remark}

\newcommand{\cic}{\bm}
\usepackage{scalerel} 
\DeclareMathOperator*{\Expectation}{\scalerel*{\mathbb{E}}{\textstyle\sum}}
\def\d{{\rm d}}

\usepackage{hyperref} 
\hypersetup{
    colorlinks=true,       
    linkcolor=blue,          
    citecolor=magenta,        
    filecolor=magenta,      
    urlcolor=cyan           
}

\begin{document}
\title[Prime Wiener-Wintner Theorem]{The Wiener Wintner Theorem Along the Primes}

\author{Jan Fornal}
\address{
Department of Mathematics,
University of Bristol \\
Beacon House, Queens Rd, Bristol BS8 1QU}
\email{nc24166@bristol.ac.uk}

\author{Anastasios Fragkos}
\address{
Department of Mathematics,
Georgia Tech \\
686 Cherry Street, Atlanta, GA 30332-0160 USA}
\email{anastasiosfragkos@gatech.edu}

\author{Ben Krause}
\address{
Department of Mathematics,
University of Bristol \\
Beacon House, Queens Rd, Bristol BS8 1QU}
\email{ben.krause@bristol.ac.uk}

\author{Michael Lacey}
\address{
Department of Mathematics,
Georgia Tech \\
686 Cherry Street, Atlanta, GA 30332-0160 USA}
\email{lacey@math.gatech.edu}

\author{Hamed Mousavi}
\address{
Department of Mathematics,
University of Bristol \\
Beacon House, Queens Rd, Bristol BS8 1QU}
\email{moosavi.hamed69@gmail.com}

\author{Yu-Chen Sun}
\address{
Department of Mathematics,
University of Bristol \\
Beacon House, Queens Rd, Bristol BS8 1QU}
\email{yuchensun93@163.com}

\date{\today}

\begin{abstract}
   We prove the following Wiener-Wintner Theorem along the sequence of prime times, the first extension of the Wiener-Wintner Theorem to arithmetic sequences: for every probability space, $(X,\nu)$, equipped with a measure-preserving transformation, $T : X \to X$, and every $f \in L^p(X), \ 1< p\leq \infty$, there exists a set of full probability, $X_f \subset X$ with $\nu(X_f) =1$, so that for all $\omega \in X_f$,
    \begin{align}
        \frac{1}{N} \sum_{n \leq N} e^{ 2\pi i p_n \theta} f(T^{p_n} \omega)
    \end{align}
    converges for \emph{all} $\theta \in [0,1]$; above, $\{ 2=p_1 < p_2 < \dots \}$ are an enumeration of the primes. 
    
    Our proof lives at the interface of classical Fourier analysis, combinatorial number theory, higher order Fourier analysis, and pointwise ergodic theory, with $U^3$ theory playing an important role; our $U^3$-estimates for \emph{Heath-Brown} models of the von Mangoldt function may be of independent interest.
\end{abstract}
\maketitle 
 \setcounter{tocdepth}{1}
\tableofcontents

\section{Introduction}

By a \emph{measure-preserving system}, $(X,\nu,T)$, we mean a probability space
$ (X,\nu) $
equipped with  transformation
 $T \colon X \to X$ 
which preserves measure, 
\[ \nu(T^{-1} E) = \nu(E) \text{ for all } E \subset X \text{ measurable}.\]
The Wiener-Wintner ergodic theorem \cite{WW} is a classical generalization of Birkhoff's Theorem \cite{Birkhoff1931}.

\begin{theorem}[Wiener-Wintner Ergodic Theorem]\label{t:WW}
Let $(X,\nu,T)$ be a measure-preserving system, and let $f \in L^1(X)$ be arbitrary. Then there exists a subset $X_f \subset X$ with $\mu(X_f) = 1$ so that for all $x \in X_f$
\[ \lim_{N\to\infty} \frac{1}{N} \sum_{n \leq N} \phi( n ) f(T^n x)
\]
exists for all continuous $1$-periodic $ \phi \colon \mathbb R\to\mathbb C $.
\end{theorem}

By the Weierstrass Approximation Theorem, it suffices to prove the statement above for the exponential functions $ \theta \to e^{2\pi i k\theta}$, for integers $k$.  
With these functions in mind, the weaker statement that the limit holds for almost every $\theta $ follows from Birkhoff's Ergodic Theorem applied to 
the system $X$ times the circle group.  
Thus, the Theorem above is strictly stronger than Birkhoff's Theorem.

\smallskip 
The purpose of this paper is to extend Theorem \ref{t:WW} to averages over the prime integers. 

\begin{theorem}\label{t:primes}
Let $\mathbb{P} \coloneqq  \{ p_1 = 2 < p_2 < \dots \}$ the prime integers listed in increasing order. 
For any measure preserving system 
 $(X,\nu,T)$ and any $f \in L^\infty(X) $,  
there is an $X_f\subset X$ of full measure so that for all 
$x\in X_f$, these two conclusions hold: 
    \begin{align} \label{e;primeWW}
      \lim_{N\to \infty}  \frac{1}{N} \sum_{n \leq N} \phi(p_n) f(T^{p_n} x) \qquad \textup{exists}  
    \end{align}
    for all continuous $1$-periodic $ \phi \colon \mathbb R \to\mathbb C $. 
\end{theorem}

Concerning the restriction that $f\in L^\infty$, 
for the Wiener-Wintner statement, it can be relaxed to $f\in L^p(X)$, for $p>1$. But not $p=1$ due to \cite{MR2788358}, with further work in this direction in \cite{MR4242902}.  

\medskip 
Many extensions of the Wiener-Wintner Theorem are known, 
  with the exponential term being replaced by exponentials evaluated along  polynomials \cite{MR1257033}, nilsequences \cite{MR2544760}, Hardy field functions \cite{2944094}, etc.; see \cite{MR1995517} for a fuller discussion.

\bigskip 

Our proof of the main Theorems invokes
approximations of the von Mangoldt function, 
 higher order Fourier analysis, and structural aspects of measure preserving systems.   
Among the different possible approximations to the von Mangoldt function 
$\Lambda$ now available, we prefer the Heath-Brown approximate, or model $\Lambda _{\textup{HB}}$. 
The latter is given by the fixed complexity term 
\begin{align}\label{e:lambdaQ}
    \Lambda_{Q}(n) &\coloneqq  \sum_{Q/2 < q \leq Q} \frac{\mu(q)}{\phi(q)} c_q(n),
\\ 
\intertext{where $c_q(n)=\sum_{(a,q)=1}e^{2\pi i a/q}$ is the Ramanujan function, $\mu (q)$ is the M\"obius function, and $\phi (q) $ is the totient function. 
And, we use the sum up to a given complexity level,} 
\label{e:lambda000}
\Lambda_{\leq T} &\coloneqq  \sum_{Q \leq T} \Lambda_Q,
\end{align}
where the sum runs over dyadic $Q$.  
This was introduced by Heath-Brown \cite{MR834356}*{pg. 47} in order to `copy $\Lambda (n)$ in its distribution over arithmetic progressions.' 
A first crucial estimate is to show that $\Lambda_{\leq T}$ does much more. It is close to $\Lambda $ in the $U^3$ norm. 
\begin{align}
\lVert  \Lambda - \Lambda _{\leq\textup{exp}( (\log N) ^{1/10})} \rVert_ {U^3[N]} \lesssim (\log N) ^{-A},  \qquad A>10.
\end{align}
The norm of the left is the normalized Gowers $U^3$ norm, restricted to the integers $[N]=\{1,2, \ldots, N\}$, which is the  widely used higher order Fourier analysis norm. 
This is a sophisticated bound, that compares to similar estimates for the Cram\'er model of the von Mangoldt function \cite{MR4875606}.  
It begins by adding a possible Siegel zero correction  to $\Lambda_{\textup{HB}}$ to define $\Lambda_{\textup{HB}}'$. 
The proof is based upon inverse $U^3$ Theorems. 
The latter arose out of work of Sanders \cite{MR2994508}, 
with further elaboration by Green and Tao \cite{MR2651575}.
 We appeal to  the recent formulation  of the $U^3$ inverse theorem due  J. Leng \cite{leng2023improvedquadratic}. 
 The structure of the von Mangoldt function and its approximate allow us to see that 
 the estimate above failing leads to a contradiction.

The next crucial estimate concerns $\Lambda_Q$ defined in \eqref{e:lambda000}. 
It is that part of the Heath-Brown model with the complexity of the rationals held fixed.  
There is a remarkably small estimate of its $U^3$ norm: 
\begin{align}
\lVert   \Lambda_{Q} \rVert _{U^3[N]} \lesssim Q ^{-3/8+o(1)}, 
\qquad N > Q ^{20}. 
\end{align}
This is established in Proposition \ref {p:Jan}, and follows from a sequence of elementary, but not obvious, observations. 
It is noteworthy that no such fixed complexity bound is possible for e.g.{} the Cramer approximate to the von Mangoldt function. 
Also noteworthy, 
 the weaker bound  on the $U^3$ norm of at most $Q ^{-1/4} $ is easily accessible.
 Yet, for our proof an improvement beyond $1/4$ is essential. 
 The $3/8 -o(1)$ bound follows from a subtle additional observation.

 With these ingredients, the additional facts needed are two fold.  We need Gowers norms inequalities that relate Wiener-Wintner  averages, formed with respect to some weight, in terms the $U^3$ norm of the weight. 
 These are easily accessible, though not well represented in the literature.  
 The second fact concerns the structure of dynamical systems. There is a natural `uniform' formulation of the usual Wiener-Winter that is important for us. 
 Fix an ergodic system $(X,\mu , T)$, and $f\in L^\infty (X)$ that is weakly mixing. That is, we have 
 \begin{equation}
  \label{e;weaklyMixing} 
  \lim_{N\to\infty} 
 \frac{1}{N}  \sum_{n=1}^N 
  \lvert \langle f, T^n f \rangle \rvert=0. 
  \end{equation}
   Then, there is a   full measure set $X_f \subset X$  so that 
   for all $x\in X_f$, we have 
    \begin{align} \label{e;uniformWW}
  \lim_N \frac{1}{N} 
  \sup_\theta \Bigl\lvert 
  \sum_{n \leq N} e^{2\pi i n \theta} f(T^n x)
  \Bigr\rvert =0. 
 \end{align}
 This in fact was a key input to Bourgain's Double Recurrence Theorem \cite{MR1037434}.
%

\subsection{Acknowledgements}

We are grateful to James Leng for generously explaining some of the higher order Fourier analysis arguments.

\subsection{Notation}\label{ss:not}
We use $e(t) \coloneqq  e^{2 \pi i t} $, with $t\in \mathbb{R} $ 
throughout to denote the complex exponential. 
Let $\mu $ denote the M\"{o}bius. Namely 
\begin{align}
\mu (n) = 
\begin{cases}
1 & n=1 
\\ 
(-1) ^{k} & \textup{$n$ is divisible by $k$ distinct primes} 
\\ 
0  & \textup{$n$ is divisible by a square integer } >1 
\end{cases}
\end{align}
And $\phi (n) = \lvert \{ 1\leq a < q \colon (a,q)=1\} \rvert$ be the  totient functions.  And, the Ramanujan sum is 
\begin{align}
    c_q(n) \coloneqq \sum_{(a,q)=1} e(- an/q). 
\end{align}
 The set of primes is denoted by  $ \mathbb{P} \coloneqq  \{2<3< \dots \}$, and the von Mangoldt function is 
\begin{align}
    \Lambda(n) \coloneqq  \begin{cases} \log p & \text{ if } n = p^\alpha \text{ is a power of a prime}  \\
    0 & \text{ otherwise} \end{cases}. 
\end{align}

The initial intervals of natural numbers are denoted by $[N]=\{1,2, \ldots, N\}$. 
Averages of functions $f$ supported on $[N]$ are written as 
\begin{equation}
    \Expectation_{n\in[N]} f(n) =  \frac{1}{N} \sum_{n=1}^N f(n). 
\end{equation}

We will make use of the modified Vinogradov notation. We use $X \lesssim Y$ or $Y \gtrsim X$ to denote
the estimate $X \leq CY$ for an absolute constant $C$ and $X, Y \geq 0.$  If we need $C$ to depend on a
parameter, we shall indicate this by subscripts, thus for instance $X \lesssim_p Y$ denotes the estimate $X \leq C_p Y$ for some $C_p$ depending on $p$. We use $X \approx Y$ as shorthand for $Y \lesssim X \lesssim Y$. We use the notation $X \ll Y$ or $Y \gg X$ to denote that the implicit constant in the $\lesssim$ notation is extremely large, and analogously $X \ll_p Y$ and $Y \gg_p X$.

We also make use of big-Oh and little-Oh notation: we let $O(Y)$  denote a quantity that is $\lesssim Y$ , and similarly
$O_p(Y )$ will denote a quantity that is $\lesssim_p Y$; we let $o_{t \to a}(Y)$
denote a quantity whose quotient with $Y$ tends to zero as $t \to a$ (possibly $\infty$), and
$o_{t \to a;p}(Y)$
denote a quantity whose quotient with $Y$ tends to zero as $t \to a$ at a rate depending on $p$.

\subsection{Gowers Norms} 
\label{sub:gowers_norms}
Let $f \colon \mathbb{Z} \to \mathbb{C}  $ be a finitely supported function on the integers. Set the conjugation-difference operator to be 
\begin{align}
  \triangle_h f(x) \coloneqq  f(x) \overline{f(x+h)}, 
  \qquad x,h\in \mathbb{Z} . 
\end{align}
The basic fact here is 
\begin{align} \label{e;double}
\Bigl\lvert \sum_x f(x)\Bigr\rvert ^2 
= 
\sum_{x,h} \triangle_h f(x). 
\end{align}
The higher order conjugation-difference operator is inductively defined to be 
\begin{align}
   \triangle_{h_1,\dots,h_s}f(x) \coloneqq  \triangle_{h_s}(\triangle_{h_1,\dots,h_{s-1}} f)(x), 
    \qquad x,h_1, \ldots, h_s\in \mathbb{Z} . 
\end{align}
Then, the $s$th order Gowers norm is 
\begin{align} \label{e;GowersDef}
    \| f \|_{U^s(\mathbb{Z})}^{2^s} \coloneqq  
    \sum_{x,h_1,\dots,h_s} \triangle_{h_1,\dots,h_s}f(x). 
\end{align}
For $s=1$, this is a semi-norm, while higher orders are norms. In particular, for $s=2$, we have 
\begin{align}  \label{e;U2-ell4}
\lVert f \rVert_{U^2(\mathbb{Z} )} ^{4} 
= \int_{\mathbb{T} } \lvert  \widehat f (\theta )\rvert ^{4}\;d \theta , 
\end{align}
where $\widehat f (\theta ) = \sum_x f(x)e(-\theta x)$ is the Fourier transform of $f$.  
Moreover, we have the inductive relationship between norms given by 
\begin{gather}
\| f \|_{U^{s+1}(\mathbb{Z})}^{2^{s+1}} = \sum_{h_1,\dots,h_{s-1}} \| \triangle_{h_1,\dots,h_{s-1}} f \|_{U^2(\mathbb{Z})}^4,
\\\intertext{so in particular}
\label{e:U2U3}
    \| f \|_{U^3(\mathbb{Z})}^{8} = \sum_{h} \| \triangle_{h} f \|_{U^2(\mathbb{Z})}^4.
\end{gather}

We recall the following fundamental inequality for Gowers norms:
\begin{lemma}[Gowers-Cauchy-Schwarz Inequality]\label{CauchySchwarzGowers}
	For $(f_\omega)_{\omega \in \{0, 1\}^s} \colon G \to \mathbb{C}^{\{0, 1\}^s}$, we define 
	$$\langle (f_\omega)_{\omega \in \{0, 1\}^s} \rangle_{U^s(\Z)} := \sum_{n, h_1, \ldots h_s \in \Z} \prod_{\omega \in \{0, 1\}^s} C^{|\omega|}f(x + \omega \cdot h),$$ where $Cg:=\overline{g}$ is complex conjugation. Then
	$$|\langle (f_\omega)_{\omega \in \{0, 1\}^s} \rangle_{U^s(\Z)} | \le \prod_{\omega \in \{0, 1\}^s} \|f_\omega\|_{U^s(\Z)}.$$ \end{lemma}

For integers $H$, let $[H]=\{0,1, \ldots, h-1\}$. 
Denote the usual expectation by 
\begin{align}
 \Expectation_H f(h) =    \Expectation_{h \in [H]} \, f(h) \coloneqq  \frac{1}{H} \sum_{a \in [H]} f(a)
\end{align}
We then define normalized $U^3$ norms as follows. For $f \colon [N] \to \mathbb C$, 
set 
\begin{equation}
    \lVert f \rVert_{U^3[N] } ^8 \coloneqq  
    \frac{\| f \|_{U^3(\mathbb{Z})}^{8} }{\lVert \mathbf{1}_{[N]} \rVert_{U^3(\mathbb Z) } ^8}.  
\end{equation}
Note that  $\lVert \mathbf{1}_{[N]} \rVert_{U^3(\mathbb Z) } ^8 \simeq N^4$. We use normalized $\ell^p$ norms,
\begin{align}
    \| f \|_{L^p([N])}^p := \Expectation_{[N]} |f(n)|^p,
\end{align}
and for functions defined on a metric space $(X,d)$, we define the Lipschitz constant
\begin{align}\label{e:lip}
    \| F \|_{\text{Lip}} := \sup_{x \neq y} \frac{|F(x) - F(y)|}{d(x,y)}.
\end{align}
These are the general inequalities for the Gowers norms that we need.

\begin{proposition} 
Fix an integer $N$, and let $w \colon [N]\to \mathbb{C}$ be a weight, 
and $f \colon [N] \to \mathbb{C}$ a function bounded by one. 
The three inequalities below hold. 
\begin{gather} \label{e;U2control}
\Expectation_{x\in[2N]} \Bigl\lvert \Expectation_{N} f(x-n) w(n)    \Bigr\rvert ^2  
\lesssim   \lVert w \rVert _{U^2[N]} ^{2} , 
\\ 
\intertext{for $\theta  \colon [N] \to [0,1]$,}
  \label{e;U3control}
\Expectation_{x\in[2N]} \Bigl\lvert
 \Expectation_N w(n)  f(x-n)  e(\theta (x)n)  \Bigr\rvert^{4}
\lesssim  
 \lVert w \rVert_{U^3[N]} ^4  
\\   
\end{gather}
\end{proposition}
The first is well known. The second is directly relevant to the Wiener-Wintner Theorem,  
by choosing $\theta(x)$ to the value of $\theta $ that gives the largest average. 
For us, it is important that the fourth power appears above.

\begin{proof}
The first is well-known, and recorded for clarity.  
We have estimating the $L^2$ norm of a convolution, so by Plancherel, Cauchy-Schwarz and \eqref{e;U2-ell4}, 
\begin{align}
\Expectation_{x\in [2N]} 
\Bigl\lvert \Expectation_N f(x-n) w (n) \Bigr\rvert ^2  
& \leq  \frac1{2 N^3}
\sum_x \Bigl\lvert \sum_n f(x-n) w (n) \Bigr\rvert ^2 
\\ & 
\leq \frac1{2 N^3} 
\int _{\mathbb{T} } \lvert  \widehat f (\beta ) 
\widehat w (\beta ) \rvert ^2  \;d \beta 
\\ 
& \leq \frac1{2N^3} \lVert  \widehat f \rVert_4 ^2 
\lVert  \widehat w \rVert_4 ^2 
\\ 
&\lesssim  \frac1{N^{3/2}} \lVert \widehat f \rVert_4 ^2  \lVert w \rVert _{U^2[N]} ^{2}  
\\ 
& \lesssim  \lVert w \rVert _{U^2[N]} ^{2}   . 
\end{align}
The last step uses the Hausdorff-Young  inequality to bound 
the $\ell^4$ norm of  $\widehat f$, as $f$ is bounded by $1$.  

\smallskip 
For \eqref{e;U3control},  
use \eqref{e;double} to expand one square.  
\begin{align}
\Expectation_{x\in [2N]} 
\Bigl\lvert  &\Expectation_N w (n) f(x-n)  e(\theta(x)n) \Bigr\rvert ^4
\\& = 
\Expectation_{x\in [2N]} 
\Bigl\lvert \Expectation_{h\in [-N,N]} \Expectation_N 
 \triangle  _h (w (n)   f(x-n)  )
\triangle  _h e(\theta (x)n) \Bigr\rvert ^2
\\ 
& \leq 
\Expectation_{x\in [N]} \Expectation_{h\in [-N,N]}
\Bigl\lvert  \Expectation_N 
\triangle  _h w (n)  \cdot \triangle  _h f(x-n)    \Bigr\rvert ^2 .
\end{align}
The second step follows as $\triangle  _h e(\theta (x)n)
= e(\theta (x) h)$ is not a function of $n$.  
But, then we are free to use our first inequality. 
For each $h$, $\triangle  _h f $ is bounded by $1$, hence 
\begin{align}
\Expectation_{x\in [2N]} \Expectation_{h\in [-N,N]}
\Bigl\lvert  \Expectation_N 
\triangle  _h w (n) 
\triangle  _h f(x-n)       \Bigr\rvert ^2
\lesssim   
\Expectation_{h\in [-N,N]}
\lVert \triangle  _h w \rVert_{U ^{2}[N]} ^{2} 
\lesssim  \lVert w \rVert _{U ^{3}[N]} ^{4}. 
\end{align}

\end{proof}

 


\section{$U^3$ Norms of the Heath-Brown Model, with Fixed Complexity }\label{s:Jan}

In this section we establish our main combinatorial result, concerning that part of the Heath-Brown model, in which the complexity of the rationals is held fixed. Recall the definition of $\Lambda _Q$ in \eqref{e:lambdaQ}. 
\begin{proposition}\label{p:Jan}
	Suppose that $M \geq Q^{20}$. Then for any $\epsilon >0$
	\begin{align}\label{e:LambdaQ}
		\| \Lambda_Q \|_{U^3([M])} \coloneqq  
		\Bigl\| \sum_{ Q/2\leq  q  < Q, } \frac{\mu(q)}{\phi(q)} c_q \Bigr\|_{U^3[M]} \lesssim_\epsilon Q^{\epsilon-  3/8}.
	\end{align}
\end{proposition}

\begin{remark}
    The proof is entirely elementary, and goes beyond a dimension counting type bound.  A bound better than $Q ^{-1/4}$ is needed for the main results of this paper, and a similar approach yields the more general estimate:
    \begin{align}
        \| \Lambda_Q(m) \|_{U^s([M])} \lesssim_{\epsilon} Q^{\epsilon - s/2^s}
    \end{align}
    whenever $s \geq 2$ and $M \geq Q^{2^{s+1} + s + 1}$. The proof of this inequality will appear in the above-mentioned forthcoming work of the first and third authors.
\end{remark}

\begin{proof}
	Recall the definition of Gowers $U^3(\Z)$-norm  given in \eqref{e;GowersDef}, and the asymptotic $  \| \mathbf{1}_{[M]} \|^8_{U^3(\Z)} \approx M^4$.  
	The proposition reduces to the following:
	\begin{align} \label{e;fixedScaleToProve}
		\| \Lambda_Q \mathbf{1}_{[M]}\|_{U^3(\Z)}^8
		\lesssim_\epsilon   M^4 Q^{-3 + \epsilon  }, \qquad \epsilon   >0, \ M > Q^{20}. 
	\end{align}
	The expansion of this term, including that of the Ramanjuan sum $c_q$, leads to 
	tuples of integers  $(x, h_1, h_2, h_3) \in [M]$, 
	a selection of square free integers $ Q/2\leq q_\omega <Q$, for $\omega \in \{0,1\}^3$,  and integers $a_\omega $ with $(a_\omega , q_\omega )=1$. 
	These are the implicit assumptions for the remainder of the proof. 
	The Gowers norm is 
	\begin{align} 
		\| \Lambda_Q \mathbf{1}_{[M]} \|_{U^3(\mathbb{Z})} ^8 &= \sum_{\substack{ \mathbf{q}=(q_\omega)_{\omega \in \left\{0,1\right\}^3} \\ \mathbf{q}  \in \left[Q/2,Q\right)^{8} } } \sum_{\substack{(a_\omega)_{\omega \in \{0, 1\}^3} \\ (a_{\omega},q_{\omega})=1 }} \prod_{\alpha \in \{0,1\}^3}\frac{\mu(q_\alpha)}{\phi(q_\alpha)} 
		\\\label{expansion-u3-norm}
		& \qquad  \times \sum_{x,h_1,h_2,h_3} e(- x L_0((a_\omega/q_\omega)_{\omega \in \{0, 1\}^3 }))
		\\ &\qquad\times \prod_{i=1}^3 e(- h_i L_i((a_\omega/q_\omega)_{\omega \in \{0, 1\}^3 })) \cdot \mathbf{1}_{x + \alpha \cdot h \in [M]}. 
	\end{align}
	The $L_i$, for $i=0,1,2,3$, are linear functions of the $a_\omega /q_\omega $ are defined by 
	\begin{gather} \label{cube-linear-form}
		L_0((a_\omega/q_\omega)_{\omega \in \{0, 1\}^3 }) = \sum_{\omega \in \{ 0, 1 \}^3} (-1)^{|\omega|} \frac{a_\omega}{q_\omega},
		\\
		\intertext{and for $i \in \{1, 2, 3\}$:}
		\label{cube-faces-linear-form}
		L_i((a_\omega/q_\omega)_{\omega \in \{0, 1\}^3 }) = \sum_{\substack{\omega \in \{ 0, 1 \}^3 \\ \omega_i = 1}} (-1)^{|\omega|} \frac{a_\omega}{q_\omega}
	\end{gather} 
	
	We first show that a significant amount of ``diagonalization" occurs.  
	The sum in \eqref{expansion-u3-norm} is at most $O(M^3 Q^{16})$, 
	summing over those choices of   fractions 
	$({a_\omega}/{q_\omega})_{\omega \in \{0, 1\}^3}$ 
	where one of the linear functions 
	\[ (L_i((a_\omega/q_\omega)_{\omega \in \{0, 1\}^3 })), \; \; \; 0 \leq i \leq 3 \] 
	is noninteger. 
	Indeed, suppose for concreteness that $L_0$ is noninteger, the case of $L_1, L_2$ or $L_3$ can be dealt very similarly but are slightly simpler; one gets 
	\[ \|L_0((a_\omega/q_\omega)_{\omega \in \{0, 1\}^3 }) \|_{\mathbb{T}} \geq \frac{1}{Q^8}. \] 
	Since, for fixed $h_1, h_2, h_3 \in \Z$, the set of $x$ such that 
	\[ x + \sum_{i = 1}^3 \alpha_i h_i \in [M]\] 
	is an interval, the inner sum over $x$ of \eqref{expansion-u3-norm} is 
	just a geometric series, contributing $O(Q^8)$. The number of possible choices for $h_1, h_2, h_3$ is then $O(M^3)$, so eventually the contribution is of the order of most:
	\begin{align} \label{expansion-non-diagonal}
	\sum_{\substack{ \mathbf{q}=(q_\omega)_{\omega \in \left\{0,1\right\}^3} \\ \mathbf{q}  \in \left[Q/2,Q\right)^{8} } } \sum_{\substack{(a_\omega)_{\omega \in \{0, 1\}^3} \\ (a_{\omega},q_{\omega})=1 }} \prod_{\alpha \in \{0,1\}^3}\frac{1}{\phi(q_\alpha)} M^3 Q^8 
		\lesssim M^3 Q ^{16}. 
	\end{align}
	Above, summing all $a_\omega$ that are coprime to $q_\omega$ cancels the totient $\phi(q_\omega)$. This leaves the possible choice of $q_\omega $, which is at most  $Q^8$. 
	Recalling that we assume $M>Q^{20}$,  this term is smaller than \eqref{e;fixedScaleToProve}, the bound to establish.

	\smallskip 
	
	The remaining contribution to \eqref{expansion-u3-norm} is 
	obtained by adding the condition that each of the $L_j$ evaluate to an integer: 
	\begin{align}
		\sum_{\substack{(a_\omega,q_\omega)_{\omega \in \{0, 1\}^3} \\ L_j((a_\omega/q_\omega)_{\omega \in \{0, 1\}^3 }) \in \Z, \ 0 \leq  j \leq 3}} \prod_{\alpha \in \{0,1\}^3}\frac{\mu(q_\alpha)}{\phi(q_\alpha)} \sum_{x, h_1, h_2, h_3 \in \Z} \mathbf{1}_{x + \alpha \cdot \mathbf{h} \in [M]}. 
	\end{align}
	And, the last sum above is at most     
	$\|\mathbf{1}_{[M]}\|_{U^3(\Z)} \simeq M ^{4}$. 
	The claim of the proposition reduces to the following inequality:
	\begin{align} \label{expansion-u3-norm-2}
		\sum_{\substack{(a_\omega,q_\omega)_{\omega \in \{0, 1\}^3}  \\ L_j((a_\omega/q_\omega)_{\omega \in \{0, 1\}^3 }) \in \Z, \ 0 \leq  j \leq 3}} \prod_{\alpha \in \{0,1\}^3}\frac{\mu^2(q_\alpha)}{\phi(q_\alpha)} \lesssim_\epsilon   Q^{-3 + \epsilon  }, \qquad \epsilon >0. 
	\end{align}

	We identify  new divisibility constraints. 
	Fix, for the moment, the tuple of denominators $(q_\omega)_{\{0,1\}^3}$ and denote the product of the $q _ \omega $ by 
	\[ R = R\big( (q_\omega)_{\{0,1\}^3} \big) 
	= \prod _{\omega \in \{0,1\}^3} q_\omega . 
	\] 
	We establish momentarily that $ \text{rad}(R)^4 \mid  R$, where 
	$\text{rad}(R)$ is the radical of the number $R$, thus
	\begin{align}
		\text{rad}\Bigl(\prod_{i} p_i^{e_i}\Bigr) = \prod_i p_i, \; \; \; p_i \in \mathbb{P}, \ e_i \geq 1.
	\end{align}
	Also set $ \triangle _{(q_\omega)_{\{0,1\}^3}}$ to be  the number of tuples $(a_\omega)_{\{0,1\}^3}$ such that the function $L_j((a_\omega/q_\omega)_{\omega \in \{0, 1\}^3 }) \in \mathbb{Z}$ are satisfied for all $0 \leq j \leq 3$. We assert  that:
	\begin{align} \label{counting-possible-numerators}
	 \triangle _{(q_\omega)_{\{0,1\}^3}} \leq 
	\prod_{p :  v_p(R)=4} (p-1) \times \prod_{p \colon v_p(R) \geq 5} (p-1)^{v_p(R)-4}. 
	\end{align}
	Above, $v_p(R)$ is the highest power of $p$ that divides $R$, and we know that if $p\mid R$, then  $v_p(R)\geq 4$.

	Indeed, since  the linear forms evaluate to integers, namely 
	\eqref{cube-linear-form} and \eqref{cube-faces-linear-form} are integers, one concludes also that all of these sums are integers: 
	\begin{align} \label{cube-faces-linear-form-2}
		\sum_{\substack{\omega \in \{ 0, 1 \}^3 \\ \omega_i = j}} (-1)^{|\omega|} \frac{a_\omega}{q_\omega} \in \Z, 
		\qquad i \in \{1, 2, 3 \}, \ j \in \{0, 1 \}.
	\end{align}
	And, given the presence of the M\"{o}bius function, 
	each $q_\omega$ is square free. 
	And so, each fraction $\frac{a_\omega}{q_\omega}$ can be uniquely written as the sum:
	\begin{align}
		\frac{a_\omega}{q_\omega} \equiv \sum_{p \in \mathbb{P}} \frac{a_{\omega p}}{p} \pmod 1
	\end{align}
	where $a_{\omega p} \in \{0, \ldots, p-1 \}$ is nonzero if and only if $p | q_\omega$. The situation that \eqref{cube-faces-linear-form-2} is an integer happens if and only if for each prime $p$, $i \in \{1, 2, 3\}$ and $j \in \{0, 1\}$ one has:
	\begin{align} \label{cube-faces-linear-form-3}
		p  \mid \sum_{\substack{\omega \in \{ 0, 1 \}^3 \\ \omega_i = j}} (-1)^{|\omega|} a_{\omega p}.
	\end{align}
	We now combinatorially encode the above constraint on to the vertices of the cube $\{0,1\}^3$: 
	
	\smallskip

	Given a prime $p$, mark all vertices $\omega$  on the  cube 
	$\{0,1\}^3$ for which $a_{\omega p} \not \equiv 0 \pmod p$ (in other words $p | q_\omega$). 
	Notice that a face of the cube is specified by fixing one of the three  coordinates of $\omega $, which is the condition imposed in \eqref{cube-faces-linear-form-3}. 
	From the above equation one sees that for each face there is either zero or at least two marked vertices; this forces the number of vertices  marked by a prime, on the entire cube, to be at least four, as it is impossible to mark one, two, or three vertices of a cube without leaving  a face with only one marked vertex. 
	This means that for each prime $p$ dividing $R$, necessarily $p^4$ divides $R$. 
	That is, our first assertation that $\text{rad}(R)^4 \mid R$ holds.  This observation leads to the $Q ^{-1/4}$ bound, yet 
	we need to exceed this bound.

	The second assertation \eqref{counting-possible-numerators} 
	will supply the $Q ^{-3/8}$ bound.  Its verification requires a combinatorial analysis postponed until after the proof.  
	
	\medskip 
	
	It remains to provide the bound in \eqref{expansion-u3-norm-2}. 
	The left-hand side of \eqref{expansion-u3-norm-2} can be estimated by:
	\begin{align} \label{expansion-u3-norm-3}
		\sum_{\substack{(q_\omega)_{\omega \in \{0, 1\}^3}}} \triangle _{(q_\omega)_{\{0,1\}^3}} \prod_{\alpha \in \{0,1\}^3}\frac{\mu^2(q_\alpha)}{\phi(q_\alpha)}.
	\end{align}
	The product of fractions $\frac{\mu^2(q_\alpha)}{\phi(q_\alpha)}$ can be written purely in terms of $ R = R\big( (q_\alpha)_{\{0,1\}^3} \big):$
	\begin{align}
		\prod_{\alpha \in \{0,1\}^3}\frac{\mu^2(q_\alpha)}{\phi(q_\alpha)} = \prod_{p \mid R} \frac{1}{(p-1)^{v_p(R)}}
	\end{align}
	and, with the crucial estimate \eqref{counting-possible-numerators}, this reduces \eqref{expansion-u3-norm-3} to the following:
	\begin{align} 
		\sum_{\substack{(q_\omega)_{\omega \in \{0, 1\}^3} \\ 
				Q/2\leq q_\omega <Q}} \frac{1}{\phi(\text{rad}(R))^3} \prod_{p^5 \mid R} \frac{1}{(p-1)} , 
	\end{align}
	and above we have recalled the restriction on the size of the $(q_\omega )$ from the definition \eqref{e:lambdaQ}.

	The number of possible choices of $(q_\omega)_{\omega \in \{0, 1\}^3}$ so that their product is $R$ is at most $\tau(R)^8$, where $\tau $ is the divisor function.  
	So, the above expression is bounded by:
	\begin{align} \label{expansion-u3-norm-4}
		 \sum_{\substack{(Q/2) ^{8} \leq R < Q^8 \\ \text{rad}(R)^4 \mid R\\
				R\mid \text{rad}(R)^8}} \frac{\tau(R)^8}{\phi(\text{rad}(R))^3} \prod_{p^5 \mid R} \frac{1}{(p-1)} .
	\end{align}
	For any possible choice of $R$, $\tau (R) ^{8}\lesssim Q ^{o(1)}$, so we can ignore that term above.  In a similar vein, note that 
	\begin{equation}
		\frac{1}{\phi(\text{rad}(R))^3} \prod_{p^5 \mid R} \frac{1}{(p-1)} 
		=
		\prod_{p \mid R} \frac{1}{(p-1)^3} 
		\prod_{p^5 \mid R} \frac{1}{(p-1)} 
		\lesssim Q ^{o(1)} 
		\prod_{p \mid R} \frac{1}{p^3} 
		\prod_{p^5 \mid R} \frac{1}{p}.  
	\end{equation}
	So, it suffices to bound 
	\begin{align}  
		\sum_{\substack{(Q/2) ^{8} \leq R < Q^8 \\ \text{rad}(R)^4 \mid R\\
				R\mid \text{rad}(R)^8}}  \text{rad}(R)^{-3}
		\prod_{p^5 \mid R} \frac{1}{p} \lesssim Q^{-3 + o(1)}  .
	\end{align}
	
	Let us explain the $Q ^{o(1)}$ term above. 
	Let $r = \text{rad}(R)$, and $R= r^4\cdot s$, 
	so that $s\mid r^4$.  Observe that for a fixed $r$ there are at most $Q^{o(1)}$ possible choices of $s$.  
	The integer $\textup{rad}(s)$ divides $r$. 
	Hence, there are at most $2^{\omega(r)}$ possible choices of $\textup{rad}(s)$, with $\omega(n)$ being the number of distinct prime factors of $n$. 
	With $\textup{rad}(s)$ fixed, there are at most 
	$4^{\omega (s)}$ possible choices for $s$, since $s\mid \textup{rad}(s)^4$.  Thus, the number of possible choices of $s$ is at most $2^{3\omega(r)}=Q^{o(1)}$.  
	
	We parameterize the sum using $r$ and $s$, as defined.  It remains to show that 
	\begin{align}  
		\sum_{\substack{ r \gtrsim Q \\   }  } 
		\mu(R)^2r ^{-3} \max_{\substack{s\geq 1\\  r^4s \approx Q^8\\
				s\mid r,\  s\mid r^4}} \textup{rad}(s)^{-1} 
		\lesssim Q^{-3}. 
	\end{align}
	Then, write $r\simeq Q ^a$, where we will take $a= j/\log Q$, for integers $\lfloor \log Q\rfloor \leq j  \leq 2\lceil \log Q\rceil$.  
	With $r\simeq Q^a$, it follows that $s \approx Q^{8-4a}$.  Hence, $\textup{rad}(s)\gtrsim Q^{2-a}$.  Then, 
	\begin{align}  
		\sum_{\substack{ r \simeq Q^a \\   }  } 
		r ^{-3}  \times Q^{a-2} 
		\lesssim Q^{-2a-2+a} = Q^{-2-a}. 
	\end{align}
	Summing over the values of $a$, we have a geometric series of ratio about $2$. And the smallest value of $a$ is $1$, hence our bound is $Q^{-3}$, as required.  
\end{proof} 

We return to the matter of establishing \eqref{counting-possible-numerators}. 
The condition \eqref{cube-faces-linear-form-3} is a combinatorial condition imposed on the $ (a _{\omega })$ on each face of the cube $\{0,1\}^3$.  
Recall that given a prime $p$, we mark  a vertex $\omega$ 
of $\{0,1\}^3$    if  $a_{\omega p} \not \equiv 0 \pmod p$ (in other words $p | q_\omega$).
Suppose that we are in the case of at least $4$ marked vertices in the cube; we denote the set of these vertices by $S = S_p$. 
A key point is that   the values of $a_{\omega p}$ 
at a small subset $T\subset S$   determines all the values of $a_{\omega p}$, 
for $\omega\in S$.  

Assume that at the beginning we know the values of $a_{\omega p}$ for the subset $T = T_p$ of marked vertices $S$ and color these vertices green. 
Now we proceed according to the following algorithm:
\begin{enumerate}
	\item Query whether there exists a face with all marked vertices but one vertex  being green. If that is \emph{not} the case, then we \textbf{quit} the algorithm. 
	
	\item If the algorithm runs, then the equation \eqref{cube-faces-linear-form-3} determines the value $a_{\omega p}$.
	
	\item That said, we color the vertex $\omega$ green and 
	we return to the initial query.
\end{enumerate}

\begin{lemma}\label{l:algorithm}
	For any collection of $S = S_p$ of marked vertices, there exists a subset $T \subset S$ of marked vertices so that 
	these two conclusions hold. 
	\begin{itemize}
		\item The following upper bound holds
		\[ |T| \leq \begin{cases} 1 & \text{ if } |S| = 4 \\
			|S| - 4 & \text{ if } 5 \leq |S| \leq 8 \end{cases};\]
		\item If each vertex in $T$ is marked green and $T$ is submitted to the above algorithm, the algorithm terminates with each vertex $S$ being marked green.
	\end{itemize}
\end{lemma}

Before turning to the proof, it is helpful to provide an example of the algorithm. 
Suppose that set $S$ is the set of all blue and green points in Figure \ref{fig1}, and we choose $T$ to be the set of green points (as in the algorithm). 
The face numbered $3$ satisfies the first step of the algorithm; therefore, the point $A$ can be colored green. Then we go back to step one of the algorithm; face $4$ satisfies the query, so we change the color of vertex $B$ to green. In the next two repetitions we repeat that procedure with faces $1$ and $5$, which leads to coloring all blue vertices green, which terminates the algorithm. 

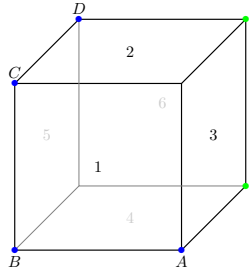
\begin{figure}
	\resizebox{.22\textwidth}{!}{%
		\begin{tikzpicture}
			\draw[thick](4,4,0)--(0,4,0)--(0,4,4)--(4,4,4)--(4,4,0)--(4,0,0)--(4,0,4)--(0,0,4)--(0,4,4);
			\draw[thick](4,4,4)--(4,0,4);
			\draw[gray](4,0,0)--(0,0,0)--(0,4,0);
			\draw[gray](0,0,0)--(0,0,4);
			\draw(2,2,4) node{1};
			\draw(2,4,2) node{2};
			\draw(4,2,2) node{3};
			\draw[gray!40](2,0,2) node{4};
			\draw[gray!40](0,2,2) node{5};
			\draw[gray!40](2,2,0) node{6};
			\foreach \Point in {(4,0,4), (0,4,4), (0,0,4), (0,4,0)}{
				\draw \Point node[blue]{$\bullet$};
			}
			\foreach \Point in {(4,4,0), (4,0,0)}{
				\draw \Point node[green]{$\bullet$};
			}
			\draw (4,0,4) node[below]{$A$};
			\draw (0,0,4) node[below]{$B$};
			\draw (0,4,4) node[above]{$C$};
			\draw (0,4,0) node[above]{$D$};
	\end{tikzpicture}}
	\caption{An example of labelling of a cube with green and blue marked vertices in the case $|S|=6$.} \label{fig1} 
\end{figure}

With this example in mind, we turn to the proof of our claim.

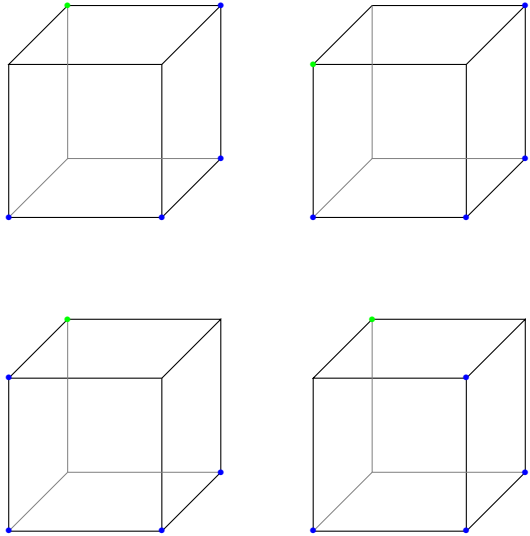
\begin{figure}
	\begin{tabular}{c@{\hspace{1cm}}c} 
		\resizebox{.2\textwidth}{!}{\begin{tikzpicture}
				\draw[thick](4,4,0)--(0,4,0)--(0,4,4)--(4,4,4)--(4,4,0)--(4,0,0)--(4,0,4)--(0,0,4)--(0,4,4);
				\draw[thick](4,4,4)--(4,0,4);
				\draw[gray](4,0,0)--(0,0,0)--(0,4,0);
				\draw[gray](0,0,0)--(0,0,4);
				\foreach \Point in {(4,4,0), (4,0,4), (4,0,0), (0,0,4)}{
					\draw \Point node[blue]{$\bullet$};
				}
				\draw (0,4,0) node[green]{$\bullet$};
		\end{tikzpicture}} &
		\resizebox{.2\textwidth}{!}{\begin{tikzpicture}
				\draw[thick](4,4,0)--(0,4,0)--(0,4,4)--(4,4,4)--(4,4,0)--(4,0,0)--(4,0,4)--(0,0,4)--(0,4,4);
				\draw[thick](4,4,4)--(4,0,4);
				\draw[gray](4,0,0)--(0,0,0)--(0,4,0);
				\draw[gray](0,0,0)--(0,0,4);
				\foreach \Point in {(4,4,0), (4,0,4), (4,0,0), (0,0,4)}{
					\draw \Point node[blue]{$\bullet$};
				}
				\draw (0,4,4) node[green]{$\bullet$};
		\end{tikzpicture}} \\[1cm]
		\resizebox{.2\textwidth}{!}{\begin{tikzpicture}
				\draw[thick](4,4,0)--(0,4,0)--(0,4,4)--(4,4,4)--(4,4,0)--(4,0,0)--(4,0,4)--(0,0,4)--(0,4,4);
				\draw[thick](4,4,4)--(4,0,4);
				\draw[gray](4,0,0)--(0,0,0)--(0,4,0);
				\draw[gray](0,0,0)--(0,0,4);
				\foreach \Point in {(0,4,4), (4,0,4), (4,0,0), (0,0,4)}{
					\draw \Point node[blue]{$\bullet$};
				}
				\draw (0,4,0) node[green]{$\bullet$};
		\end{tikzpicture}} &
		\resizebox{.2\textwidth}{!}{\begin{tikzpicture}
				\draw[thick](4,4,0)--(0,4,0)--(0,4,4)--(4,4,4)--(4,4,0)--(4,0,0)--(4,0,4)--(0,0,4)--(0,4,4);
				\draw[thick](4,4,4)--(4,0,4);
				\draw[gray](4,0,0)--(0,0,0)--(0,4,0);
				\draw[gray](0,0,0)--(0,0,4);
				\foreach \Point in {(4,4,4), (4,0,4), (4,0,0), (0,0,4)}{
					\draw \Point node[blue]{$\bullet$};
				}
				\draw (0,4,0) node[green]{$\bullet$};
		\end{tikzpicture}}
	\end{tabular}
	\caption{The four cases, up to symmetry, in which there are 5 marked vertices on the cube, with 1 being green, and no face having 4 marked vertices.}  
	\label{fig2} 
\end{figure} 

\begin{proof}[Proof of Lemma \ref{l:algorithm}]
	First of all, let $f$ denote the function $\{1,2,3,4, 5, 6, 7, 8\} \to \Z_{\geq 0}$ so that $f(n)$ represents the minimal number with the property: for each set $S$ of marked vertices in the cube  
	$\{0,1\}^3$ 
	of size $n$ (possibly having a face with only one vertex in $S$), there exists a set $T \subset S$ of size $f(n)$ of vertices that are marked green. 
	Moreover, inputing $T$  into the algorithm leads to all vertices of $S$ being colored in green. 
	
	One observes that adding one vertex outside of $S$ to both $S$ and $T$ (i.e.\ adding an additional green vertex) does not change the outcome of algorithm. 
	This means that $f(n + 1) \leq f(n) + 1$. 
	Now, $f(1) = f(2) = f(3) = 0$ as in these cases there is always face with only one vertex. 
	Thus, $f(4)\leq 1$, and in fact equals $1$. 
	So, the only equality that requires checking is $f(5) = 1$. In the case with $5$ vertices in $S$ and a face containing all four vertices one picks $T$ to be any vertex on the face containing four vertices from $S$. The case with $f(5) = 1$ is still unresolved, but now we can without loss of generality assume that bottom face contains $3$ vertices in $S$ and top face contains $2$ vertices in $S$. Up to rotation with respect to the top-bottom axis, there are only $4$ possible alignments for which the condition with $4$ vertices on a face is not satisfied. The collection of $4$ diagrams in Figure \ref{fig2} handles each case separately; elementary case analysis concludes the proof.
\end{proof}

That said, one can finally pass to the counting that leads to the estimate \eqref{counting-possible-numerators}. Since we are solving the collection of systems of linear congruences modulo each prime (i.e. \eqref{cube-faces-linear-form-3}), one derives the multiplicative form of the upper bound on $\triangle _{(q_\omega)_{\{0, 1\}^3}}$. For each prime $p$, we take the set of marked vertices $S_p$ and the appropriate subset $T_p \subset S_p$ that we constructed. There are 
\[ (p-1)^{|T_p|} \] choices for $(a_{\omega p})_{\omega \in T_p}$ as each $a_{\omega p} \neq 0 \mod p$, and then $(a_{\omega p})_{\omega \in S_p \backslash T_p}$ is uniquely determined by the construction of $T_p$ and the system of equations \eqref{cube-faces-linear-form-3}. The remaining $a_{\omega p}$ must be zero as $S_p$ captures exactly these $\omega$ with $p | q_\omega$.  This gives us the formula  \eqref{counting-possible-numerators}.

\medskip 

\section{$U^3$ Norms of the Von Mangoldt Function and the Heath-Brown Model} \label{s:Hamed} 
We establish the following $U^3$-approximation result of the von Mangoldt function by its Heath-Brown model.

\begin{proposition}\label{p:logsavings}
	Let $N \geq 1$, and $ Q = Q_N \coloneqq  \exp( (\log N)^{1/10})$. Then for  all $ A \geq 1$
	\begin{align}\label{e:logsavings}
		\| \Lambda - \Lambda_{\leq Q} \|_{U^3[N]} \lesssim_A (\log N)^{-A},  
	\end{align}
	where $\Lambda $ is the von Mangoldt function and  $\Lambda_{\leq T}$ is defined in \eqref{e:lambda000}.
\end{proposition}


\begin{remark}
    As in \S \ref{p:Jan}, one may want to establish analogous versions of this inequality for the $U^s$ norm for $s \geq 3$. All reasoning below generalises except for the fact that we are not aware whether Lemma \ref{l:siegelcomp} holds for arbitrary $U^s$-norm. To obtain the analogue of this lemma for arbitrary Gowers norms, one would need to probably refer to the very powerful result of Leng, Sah and Sawhney \cite{leng2024quasipolynomial} for $L^{p_s}$ bounded functions, where $p_s$ is some exponent depending on $s$.
    
    The reason why we need to upgrade result from \cite{leng2024quasipolynomial} is that one has poor $L^\infty$ control on Heath-Brown model of von Mangoldt function, contrary to averaged control, due to the Lemma $4.6$ from \cite{KMTT}. We leave as an open question whether one can derive the quasipolynomial inverse theorem for $U^s$ Gowers norm when function has $L^{p_s}$-norm at most $1$ (with $p_s$ probably very large).
\end{remark}

The proof of this proposition is quite involved. 
It invokes the $U^3$ inverse theorem, with important contributions by Sanders \cite{MR2994508}, and  Green and Tao \cite{MR2651575}, among other references.
We appeal to the statement of Leng \cite{leng2023efficient}.  
To use these tools, we introduce a variant of $\Lambda _{\leq Q}$ that accounts for exceptional Siegel zeros. 

\smallskip

For $q\in \mathbb{N}$, a character $\chi_q$ is a nonzero multiplicative function from $(\mathbb{Z}_q)^{\times}\to \mathbb{C}$. Define the $L-$function $L(\chi_q,s)$ to be 
\begin{align}
	L(s,\chi) \coloneqq  \sum_{n>0} \frac{\chi_q(n)}{n^s} \, \textup{ for } Re(s) > 1,
\end{align}
and apply analytic continuation to extend the definition to the whole plane. It is known, \cite{iwaniec}*{Chapter 5}, 
that $L(\chi_q,s)$ has an entire analytic continuation, except for the case of $ \chi _q $ being the principle character, in which case there is a simple pole at $s=1$.  
In the case of primitive real characters $\chi_q$, the $L-$function $L(\chi_q,s)$ may have exactly one zero $ \sigma_q$ (known as the Siegel zero), such that 
\begin{align} \label{e;whereIsSigma}
	\frac{c_\epsilon }{q^{\epsilon}}
	< 1-\sigma_q<\frac{1}{\log q} , \qquad 0<\epsilon <1 .
\end{align}
We call such character, if it exists, \emph{exceptional}, 
and set $q=q_{\textup{Siegel}}$ and set $\sigma_q$ to be the pertaining Siegel zero. 
If 
\begin{align}\label{e:sigma0} \frac{\log^2 T}{(\log \log T)^8} \leq
	q_{\textup{Siegel}} \leq   T,
\end{align} 
we say that $ \sigma \coloneqq  \sigma_{q_{\textup{Siegel}}}$ is the \emph{Siegel zero at level $T$}. 
In particular, the Siegel zero at level $T$, if it exists,  is unique by the Landau-Page Theorem, \cite{iwaniec}*{Chapter 5}, or \cite{MR4875606}*{Section 2}.

\bigskip

Next, for each $T \geq 1$, recall the Heath-Brown model 
\begin{align}
	\Lambda_{ \leq T}(n) \coloneqq  \sum_{Q \leq T, \ Q:\textup{ dyadic}} \Lambda_Q(n) = \sum_{Q \leq T, \ Q:\textup{ dyadic}} \big( \sum_{Q/2 < q \leq Q} \frac{\mu(q)}{\phi(q)}c_q(n) \big),
\end{align}
see \eqref{e:lambda000},
and in the presence of  Siegel zero at scale $T^{1/2}$, set
\begin{align}\label{e:Lambdaleq}
	\Lambda'_{ \leq T}(n) \coloneqq \Lambda_{ \leq T} (n) \cdot (1- n^{\sigma-1}\chi_{q_{\textup{Siegel}} }(n)), 
\end{align}
where $\chi_{q_{\textup{Siegel}} }$ is the exceptional character related to the Siegel zero at scale  $T^{1/2}$, with Siegel zero $\sigma$.  If there is no level-$T^{1/2}$ Siegel zero,  then set
$$\Lambda'_{ \leq T}(n) \coloneqq  \Lambda_{ \leq T}(n). $$

\bigskip




\bigskip

First we show that $\Lambda$ and $\Lambda'_{\leq Q}$ are close in $U^3$ norm. Recall that $Q = Q_N = \textup{exp}( (\log N)^{1/10} ) $.  

\begin{lemma}\label{l:siegelcomp}
	There exists  a positive constant  $\mathbf{c} > 0$ so that for each $N \gg 1$ sufficiently large, the following bound holds:
	\begin{align}
		\| \Lambda - \Lambda_{\leq Q}' \|_{U^3[N]} \lesssim \exp(- (\log N)^\mathbf{c}).
	\end{align}
\end{lemma}

The proof uses the approach in Leng \cite{leng2023efficient,leng2023improvedquadratic}, exploiting the strongest known version of quasipolynomial inverse theorem for the Gowers $U^3$-norm. The proof thus uses the language of nilsequences, which we recall here.

For any group $G$ (or vector space $X$) and any subset $H\subseteq  G$ (or  $Y\subseteq X$), we mean $\langle H\rangle$ (or $\langle Y\rangle$) to be subgroup (or subspace) generated by $H$ (or $Y$). 

\medskip

We call $G$  a \emph{nilpotent group} if $G$ has a lower central series which terminates after finitely many steps. An $s-$step nilpotent group $G$ has the following standard filtration:
\begin{align}
	\langle e_G\rangle = G_{(s)} \lhd G_{(s-1)}\lhd \cdots \lhd G_{(2)}\lhd G_{(1)} =G, 
\end{align}
where 
$$G_{(i)}\coloneqq [G,G_{(i-1)}] = \langle hkh^{-1}k^{-1}; h\in G , k \in G_{(i-1)}\rangle .$$
In general, we call $G_{\bullet}$ a filtration, if 
\begin{align}
	\langle e_G\rangle = G_{\ell} \lhd G_{\ell-1}\lhd \ldots \lhd G_{2}\lhd G_{1} =G, 
\end{align}
where $[G_i,G_j]\subseteq G_{i+j}$.

Roughly speaking, a \emph{nilmanifold} $G/\Gamma$ is  a differentiable manifold which has a nilpotent group of diffeomorphisms acting transitively on it. We need a more precise structure related to this as follows.

\begin{definition}\label{d:nilmalcevdef}
	A nilmanifold $X=G/ \Gamma$ of degree $k$, step $s$, complexity $M$, and dimension $d$ is the following data.
	\begin{enumerate}
		\item An $s-$step connected and simply connected nilpotent Lie group $G$ and a discrete co-compact subgroup $\Gamma$;
		\item A filtration $G_{\bullet} = \{G_i\}_{i=1}^{\infty}$ with $G_{i}=\langle e_G\rangle$ for $i\geq k+1$. 
		\item A basis $(X_1,\ldots,X_d)$ of $\mathfrak{g}=\log (G)$, known as Mal'cev basis having the following properties.
		\begin{itemize}
			\item For every $1\leq i,j\leq d$ we have 
			\begin{align}
				[X_i,X_j] = \sum_{k>i,k>j} a_{ijk}X_k
			\end{align}
			where $a_{ijk} = \frac{b_{ijk}}{c_{ijk}}$ and $(b_{ijk},c_{i,j,k})=1$ and $0\leq b_{ijk}< c_{ijk}<M$.
			\item If $d_i=\dim(G_i)$, then $\log(G_i)= \langle X_j; d-d_i<j\leq d\rangle$.
			\item Each element of $G$ can be uniquely written as 
			$$g=\exp(t_1 X_1)\exp(t_2X_2)\ldots \exp(t_dX_d) \,  \qquad t_1,\ldots,t_d\in \mathbb{R}.$$ 
			We denote the Mal'cev coordinate map  $\psi:G\to \mathbb{R}^d$ to be 
			$$\psi(g)=(t_1,\ldots t_d).$$
			\item $\Gamma$ is precisely all elements of $G$ such that $\psi(g)\in \mathbb{Z}^d$. 
		\end{itemize}
	\end{enumerate}
\end{definition}

Given a nilmanifold $G/\Gamma$, we say an element $g$ has  \emph{height $D$}, or \emph{$D-$rational}, if $g^D\in \Gamma$. 

For every filtration $G_{\bullet}$, the group $\textup{Poly}(\mathbb{Z},G_{\bullet})$ of polynomial sequences  $p:\mathbb{Z}\to G$ is
\begin{align}
	\textup{Poly}(\mathbb{Z},G_{\bullet})\coloneqq \langle g_i^{P_i(n)}; g_i\in G_{\deg(P_i)}, P_i\in \mathbb{R}[y]\rangle
\end{align}

We call a sequence $n\mapsto F(g(n)\Gamma)$  a \emph{nilsequence}, if $F$ is a Lipschitz function, and $g\in \textup{Poly}(\mathbb{Z},G)$. A degree $k$ nilsequence corresponds to a nilmanifold of degree $k$ and therefore is in the form $F(g_0g_1^{n} g_2^{\binom{n}{2}} \ldots g_k^{\binom{n}{k}} \Gamma  )$ where $g_i \in G_i, i \in \left\{1,\ldots,k\right\}.$  

Given a group $G$, we call $Z(G)$ to be the center of $G$ as follows:
\begin{align}
	Z(G) \coloneqq \{z\in G ; zg=gz \textup{ for all } g\in G\}.
\end{align}

Given a nilmanifold $G/\Gamma$,   a connected and simply connected subgroup $T\leq Z(G)$ such that  $\Gamma\cap T$ is co-compact in $T$, and a continuous homomorphism $\eta:T\to \mathbb{R}$ with $\eta(T\cap \Gamma)\subseteq \mathbb{Z}$, we call a function $F:G\to \mathbb{C}$ to be \emph{$T-$vertical character} with frequency $\eta$, if 
\begin{align}
	F(gx) = e(\eta(g))F(x) \qquad \textup{ for all } g\in G, x\in \Gamma.
\end{align}
Particularly, if $T=G_{(s)}$, then we call $F$ a \emph{vertical character}. 

Let $V$ be a vector space with basis $B$ and $W\leq V$ be a subspace. 
We call $W$ an \emph{$L-$rational subspace}  if there exists a basis $B'$ such that 
\begin{align}
	w_j =\sum_{i} a_{ij} v_i ,\qquad w_j\in B', v_i\in B
\end{align}
and $a_{ij} = \frac{b_{ij}}{c_{ij}}$,  $0\leq b_{ij}<c_{ij}\leq L$, and $(b_{ij},c_{ij})=1$. 

A connected and simply connected subgroup $G'$ of a nilmanifold $G/\Gamma$ given by Mal'cev basis $\mathcal{X}$ is $L-$rational, if $\mathfrak{g}'=\log(G')$ is $L-$rational with respect to $\mathfrak{g}=\log(G)$ given by basis $\mathcal{X}$.  

With these definitions in mind, we present the following  quantitative Gowers uniformity inverse theorem.

\begin{theorem}\label{inverseu3} Let $f: [N] \to  \mathbb{C}$ be a function with $ \Expectation_{n \in [N]} |f(n)|^{1024}  \leq 1$ and $\delta \in \left(0,\frac{1}{25}\right).$ Suppose that  \[ \left \| f \right \|_{U^3[N]} \geq \delta. \] Then there exists some absolute constant $C$, a degree two nilsequence $F(g(n)\Gamma)$ of complexity and $\|F \|_{\text{Lip}}$ bounded by $\exp\left( O\left ( \log^C\left(\delta^{-1}\right)\right )  \right)$ on a nilmanifold of dimension $d \leq O\left ( \log^C\left( \delta^{-1} \right)\right )$ such that \[ \left| \Expectation_{n \in [N]} f(n)\overline{F(g(n)\Gamma) }      \right| \geq \exp \left( -O\left ( \log^C\left(\delta^{-1}\right)\right ) \right). \]
	
\end{theorem}
Before we give the proof of Theorem \ref{inverseu3}, we state, for clarity, Leng's inverse $U^3$ theorem, which has the same formulation, but holds in the periodic setting.   
\begin{theorem} \cite[Theorem 8]{leng2023improvedquadratic}\label{inverseu3lengappendix}
	Let $f: \Z_N   \to \mathbb{C}$ be a function with $\Expectation_{n \in [N]} |f(n)|^{1024}   \leq 1$ and $\delta \in \left(0,\frac{1}{100}\right).$ Suppose that  \[ \left \| f \right \|_{U^3\left(\Z_N \right)} \geq \delta. \] Then there exists some absolute constant $C$, a degree two nilsequence $F(g(n)\Gamma)$ of complexity and $\|F \|_{\text{Lip}}$ bounded by $ \exp\left( O\left ( \log^C\left(\delta^{-1}\right)\right )  \right)$ on a nilmanifold of dimension $d \leq O\left ( \log^C\left( \delta^{-1} \right)\right )$ such that \[ \left| \Expectation_{n \in \mathbb{Z}_N} f(n)\overline{F(g(n)\Gamma) }      \right| \geq \exp \left( -O\left ( \log^C\left(\delta^{-1}\right)\right ) \right). \]
\end{theorem}
The proof of Theorem \ref{inverseu3} follows from the above by standard arguments; for completeness we provide a brief proof.
		\begin{proof}[Proof of Theorem \ref{inverseu3} Assuming Theorem \ref{inverseu3lengappendix}]
			Let $f: [N] \to \mathbb{C} $ be as in the premise of the Theorem \ref{inverseu3} and let $p $ be a prime in the interval $[8N,16N]$. Using Lemma \cite[Lemma B.5]{MR2651575}, see also \cite[Appendix A]{frantzikinakishostmultiplicative}, for $\Z_p$  we obtain for the function $\tilde{f}: \Z_p \to \mathbb{C}$ defined as  $\tilde{f} \coloneqq  f \cic{1}_{[N]}$, where $[N]$ is identified with a subset of $\Z_p$, with the same $L^{1024}$ norm. As $f$ has $\left \| \tilde{f} \right \|_{U^3(\Z_p )} \geq \delta' \coloneqq  c\delta $, we may apply Theorem \ref{inverseu3lengappendix} to obtain a degree two nilsequence $F(g(n)\Gamma)$ of complexity $\exp\left( O\left ( \log^C\left(\delta'^{-1}\right)\right )  \right)=\exp\left( O\left ( \log^C\left(\delta^{-1}\right)\right )  \right)$ on a nilmanifold of dimension $d \leq O\left ( \log^C\left( \delta'^{-1} \right)\right )=O\left ( \log^C\left( \delta^{-1} \right)\right )$ such that \[ \left| \left  \langle \tilde{f}(\cdot ),F(g(\cdot)\Gamma)  \right \rangle_{\Z_p}      \right| \geq \exp \left( -O\left ( \log^C\left(\delta'^{-1}\right)\right ) \right)=\exp \left( -O\left ( \log^C\left(\delta^{-1}\right)\right ) \right). \] For the interpretation of the inner product we identify $F(g(n)\Gamma)$ with its periodization of the first $N$ terms, as in \cite[Page 135]{greentao2008u3}, namely we identify $F(g(n)\Gamma)$ with the sequence $c(n)$ defined by the formula \[c(n) \coloneqq \sum_{k \in \Z} \cic{1}_{[k,k+N)}F(g(n+k)\Gamma). \] Hence, the inner product in Theorem \ref{inverseu3lengappendix} may be interpreted as the inner product with respect to the uniform measure on $[N].$ Therefore, by considering the definition of $\tilde{f}$, we obtain  the conclusion of Theorem \ref{inverseu3}.
		\end{proof}

We are now prepared to prove Lemma \ref{l:siegelcomp}.
\begin{proof}[Proof of Lemma \ref{l:siegelcomp}]
	
	We proceed by contradiction. We pick $\mathbf{c} \ll c_1 \ll 1$. Let $\delta   \coloneqq  \exp(- (\log N)^{c})$.  
	Assume that
	\begin{align}
		\|\Lambda-\Lambda'_{\leq Q}\|_{U^3([N])} \geq \delta  ;
	\end{align}
	We eventually show that   
	\begin{align}\label{e:overprog}
		 \bigl|\Expectation_{n \in \mathcal{P}} \big( \Lambda-\Lambda'_{\leq Q}\big)(n) \bigr| \gg \exp(-(\log \delta  ^{-1})^{O(1)}) \gtrsim \exp(-(\log N)^{1/20}). 
	\end{align}
	where $\mathcal P$ is an arithmetic progression with gap size $\exp((\log N)^{1/20})$. But, due to  Lemma \ref{l:APest},  the left side is at most 
	\begin{equation}
		\exp\bigl( -c(\log N)^{1/10} \bigr).
	\end{equation}
	  And so we will have a contradiction, completing the proof. The proof contains several steps. 
	Most, if not all, steps will introduce quasi-polynomial constants of the form $\exp ( (\log 1/\delta)^C)$. 
	All of them can be absorbed by taking $0< \mathbf{c}
    \ll 1$ sufficiently small. 
	
	\bigskip
	
	
	
	\textit{\textbf{Step 1:} Apply the quantitative  uniform norm inverse theorem.} Unlike the Cram\'{e}r model, $\Lambda_{\leq Q}'$ can take large values. Still, on average one has
	\begin{align}
		 \Expectation_{n \in [N]} |\Lambda'_{\leq Q}(n)|^{1024} \lesssim \Expectation_{n \in [N]} |\Lambda_{\leq Q}(n)|^{1024} \lesssim (\log Q)^{2^{1024} + 1024}
	\end{align}
	see \cite{KMTT}*{Lemma 4.6}. 
	Assume that 
	\begin{align}
		\|(\Lambda-\Lambda'_{\leq Q}) \|_{U^3[N]} > \delta,
	\end{align}
	and apply Theorem \ref{inverseu3} to conclude that 
	$\Lambda-\Lambda'_{\leq Q}$ must correlate with 
	a structured function which we now describe:
	\begin{itemize}
		\item There exist large enough constants $C_2,C_3,C_4>0$; 
		\item A nilmanifold $G/\Gamma$ of  degree $2$, dimension $D\leq (\log \delta^{-1})^{C_2}$, and  complexity $M \leq \exp((\log \delta^{-1})^{C_3})$;
		\item A Lipschitz function $F:G/\Gamma\to \mathbb{C}$ with $\|F \|_{\text{Lip}} \leq \exp((\log \delta^{-1})^{C_3})$
	\end{itemize}
	such that
	\begin{align} \label{e:LSSoutput}
		|\Expectation_{n<N} (\Lambda-\Lambda'_{\leq Q}) (n) F(g(n)\Gamma)| > \exp(-(\log \delta^{-1})^{C_4})  .
	\end{align} 
	
	\bigskip

	\textit{\textbf{Step 2:} Reducing to a $1-$dimensional vertical character.}    Next we use \cite{leng2023efficient}*{Lemma A.6} to Fourier expand the vertical torus. For every $S-$rational subgroup $H\subseteq G$, 
	there exists $C_5\gg C_3,C_4$ such that 
	\begin{align}
		F(g(n)\Gamma) = \sum_{|\xi| \leq  S^{D^{O(1)}}\exp((\log \delta^{-1})^{C_5})} F_{\xi}(g(n)\Gamma) + O(\exp(-\log(\delta^{-1})^{C_4})),
	\end{align}
	where $F_{\xi}$ is an $H-$vertical character with Lipschitz norm $S^{D^{O(1)}}\exp((\log \delta^{-1})^{C_5})$. In particular, by taking $H=G_{(2)}$ (the largest possible subgroup living inside the center), we can choose $S= M^{D^{O(1)}}$ for free \cite{leng2023efficient}*{Remark following Lemma A.6}. Straightforward computation shows that there exists $C_6\geq C_5,C_4$ such that 
	\begin{align}
		F(g(n)\Gamma) = \sum_{|\xi| \leq  \exp((\log \delta^{-1})^{C_6})} F_{\xi}(g(n)\Gamma) +O(\exp(-\log(\delta^{-1})^{C_4})).
	\end{align}
	Substituting this in \eqref{e:LSSoutput} and applying the triangle inequality yields the lower bound
	\begin{align}
		\sum_{|\xi| \leq  \exp((\log \delta^{-1})^{C_6})}\bigl|\Expectation_{n<N} (\Lambda-\Lambda'_{\leq Q}) (n) F_{\xi}(g(n)\Gamma)\bigr| > \exp(-(\log \delta^{-1})^{C_6}).
	\end{align}
	So, by the pigeonhole principle, there exists $C_7 \geq C_6$, a frequency  $\xi_0$ with $|\xi_0|<\exp((\log \delta^{-1})^{C_7})$ invoking a vertical character $\tilde{F} \coloneqq F_{\xi_0}$ with \[ \| F_{\xi_0} \|_{\text{Lip}} \lesssim \exp((\log \delta^{-1})^{C_7}) \] such that 
	\begin{align}
		\bigl|\Expectation_{n<N} (\Lambda-\Lambda'_{\leq Q}) (n) \tilde{F}(g(n)\Gamma)\bigr| > \exp(-(\log \delta^{-1})^{C_7}).  
	\end{align}
	Next we apply  \cite{leng2023efficient}*{Lemma 2.2} to change
	$C_7$ to a possibly larger constant $C_8\geq C_7$ and assume that $\bar{F}$ is  a one dimensional vertical character  with frequency $\xi$ with  $|\xi|\leq \exp((\log \delta^{-1})^{C_8})$, and  
	\begin{align}\label{e:verdim1}
		\bigl|\Expectation_{n<N} (\Lambda-\Lambda'_{\leq Q}) (n) \bar{F}(g(n)\Gamma)\bigr| > \exp(-(\log \delta^{-1})^{C_
			8}).  
	\end{align}

	\textit{\textbf{Step 3:} Discarding the case when $q_{\textup{Siegel}}$ is large. } Assume that $q_{\textup{Siegel}}\gg \exp((\log \delta^{-1})^{O(C_8)})$; we can use the estimate \eqref{e;SiegelToShow} below to estimate
	\begin{align} 
		\|\Lambda'_{\leq Q}-\Lambda_{\leq Q}\|_{U^3[N]} \lesssim q_{\textup{Siegel}}^{-1/18} \lesssim \exp((\log \delta^{-1})^{-O(C_8)}).
	\end{align}
	By the converse of the inverse theorem \cite{leng2024quasipolynomial}*{Lemma B.5}, we conclude that  for every $s$-step  nilsequence $F(g(n)\Gamma)$ described above, and in particular, $\bar{F}(g(n)\Gamma)$, we have 
	\begin{align}
		|\Expectation_{n<N} (\Lambda'_{\leq Q}-\Lambda_{\leq Q}) (n) \bar{F}(g(n)\Gamma)| \lesssim  \exp(-(\log \delta^{-1})^{O(C_8)/O(1)}) \lesssim \exp(-(\log \delta^{-1})^{O(C_8)}).
	\end{align}
	This and \eqref{e:verdim1} then imply that 
	\begin{align}\label{e:beforetypes}
		|\Expectation_{n<N} (\Lambda-\Lambda_{\leq Q}) (n) \bar{F}(g(n)\Gamma)| \gtrsim  \exp(-(\log \delta^{-1})^{C_8}),
	\end{align}
	which reduces to the case with no Siegel zero consideration; in other word as we  we will see below in Step 5, there is no  twisted type I sums in representing $\Lambda -\Lambda_{\leq Q}$. So we only care about the case when $q_{\textup{Siegel}}\lesssim \exp((\log \delta^{-1})^{O(C_8)})$.
	
	\bigskip
	
	\textit{\textbf{Step 4:} Decomposing $\Lambda$ and $\Lambda'_{\leq Q}$ into type I, twisted type I, and type II sums.}  We say that a sequence $a_d$ is \emph{divisor-bounded}, if $a_d$ obey the upper bound, 
    \[ a_d \lesssim \tau(d)^{O(1)} (\log N)^{O(1)}.\] 
	Recall from \cite{vaughan2003hardy}*{Chapter 3}   that a type I sum is of the form 
	\begin{align}
		n\mapsto \sum_{d \leq N^{2/3}} a_d \mathbf{1}_{d | n} \mathbf{1}_{[N']}(n),
	\end{align}
	where $a_d$ is a divisor-bounded and $N'\leq N$. 
	Similarly, a type II sum is of the form 
	\begin{align}
		n\mapsto\sum_{dw=n, \ d,w>N^{1/3}} a_db_w \mathbf{1}_{[N]}(n)
	\end{align}
	where $a_d,b_w$ are divisor-bounded.  
	Recently, Tao and Ter\"{a}v\"{a}inen  \cite{MR4875606}*{Proposition 7.2} defined the twisted type I sums to be of the form 
	\begin{align}
		n\mapsto \sum_{d|n,d\leq N^{2/3}} a_d\chi_{q_{\textup{Siegel}}}(n/d) \mathbf{1}_{N'}(n)  
	\end{align}
	where $a_d$ is a divisor-bounded sum. It is shown in \cite{vaughan2003hardy}*{Chapter 3} that $\Lambda$ can be written as the linear combination of type I and type II sums and a term with $L^{1}$ average of at most $O(\exp(-(\log N)^{1/2}))$. 
	
	We show that $\Lambda'_Q$ can also be written as a convex combination of  type I sums and twisted type I sums which are divisor bounded.
	Opening up the Heath-Brown model, we have that
	\begin{align}
		\Lambda_{\leq Q}(n) &= \sum_{q\leq Q} \frac{\mu(q)}{\phi(q)}c_q(n) = \sum_{q\leq Q} \frac{\mu(q)}{\phi(q)} \sum_{d|n} \mathbf{1}_{d|q} \mu(q/d)d \\
		&= \sum_{d|n,d\leq Q}\sum_{q\leq Q} \frac{\mu(q)}{\phi(q)}  \mathbf{1}_{d|q} \mu(q/d)d\\
		& = \sum_{d|n,d\leq Q}\frac{d\mu(d)}{\phi(d)}\sum_{\substack{q\leq Q/d \\ (d, q) = 1}} \frac{|\mu(q)|}{\phi(q)} =: \sum_{d\leq N^{2/3}}\alpha_d \mathbf{1}_{d|n}\mathbf{1}_{[N]}, 
	\end{align}
	where the bound $Q\leq N^{2/3}$ was used and 
	\begin{align}
		\alpha_d=  \cic{1}_{[Q]}(d) \frac{d \mu(d)}{\phi(d)}\sum_{\substack{q\leq Q/d \\ (d, q) = 1}} \frac{|\mu(q)|}{\phi(q)} \lesssim (\log\log d)\log Q \lesssim (\log N)^2.
	\end{align}
	Note that  
	\begin{align}
		\chi_{q_{\textup{Siegel}}}(n)\Lambda_{\leq Q}(n) = \sum_{d|n} (\alpha_d \chi_{q_{\textup{Siegel}}}(d)) \chi_{q_{\textup{Siegel}}}({n/d})\mathbf{1}_{[N]}
	\end{align}
	where $\alpha_d \chi_{q_{\textup{Siegel}}}(d)$ is divisor-bounded. Finally, by the fundamental theorem of calculus, we can treat 
	\begin{align} \label{int}
		\mathbf{1}_{[N]}(n) n^{\beta-1} = \int_1^N 1_{[M]}(n) (1-\beta) M^{\beta-2}\ dM + N^{\beta-1} 1_{[N]}(n).
	\end{align}
	So, $\Lambda'_{\leq Q}$ can be written as a convex combination of  type I sums and twisted type I sums.

	\bigskip
	
	\textit{\textbf{Step 5.} Reducing to type I, twisted type I, and type II terms.} Now that we understand the structure of   $\Lambda$ and $\Lambda'_{\leq Q}$ in terms of type I and type II sums, we use the triangle inequality and study them separately.  
	Use the decomposition from Step 4 to represent  
	\begin{align}
		\Lambda-\Lambda'_{\leq Q} = h_1(n)+h'(n)+h_2(n)+E(n),
	\end{align}
	where $h_1$ is type I sum, $h'$ is twisted type I,  $h_2$ is type II sum, and $\left \| E\right \|_{L^1[N]} \lesssim \exp(-(\log N)^{1/2})$.  From \eqref{e:beforetypes}, we know that 
	\begin{align}
		\bigl|\Expectation_{n<N} (h_1(n)+h'(n)+h_2(n)+E(n))  \bar{F}(g(n)\Gamma)\bigr| \gtrsim  \exp(-(\log \delta^{-1})^{C_8}).
	\end{align}
	We may discard the contribution of the error term, because  
    \begin{align}
    \left| \Expectation_{n <N} E(n) \bar{F}(g(n)\Gamma)  \right| \leq \left  \| \bar{F} \right \|_{\text{Lip}} \Expectation_{n<N}|E(n)|  &\lesssim  \exp( (\log \delta^{-1})^{C_8}) \exp(-(\log N)^{1/2}) \\
    &\lesssim \exp( -(\log \delta^{-1})^{C_8});
    \end{align}
    we pick $\mathbf{c}$ so that $\mathbf{c} C_8 \leq 1/200$ (say).
	Now we apply \cite{leng2023efficient}*{Proposition 5.1} to the function 
    \[ \tilde{F} \coloneqq   \exp(- \log(\delta^{-1})^{2C_8} ) \bar{F}.\]

The upshot is that at least one of 
\begin{align}\label{e:lbd00}
		\bigl|\Expectation_{n<N} h(n) \bar{F}(g(n)\Gamma)\bigr| \gtrsim  \exp(-(\log \delta^{-1})^{C_h C_8}), \; \; \; h \in \{ h_1,h',h_2\}
	\end{align}
where $C_h = 5$ if $h=h_1$ and $3$ otherwise;
we address each case in turn:
    
	If we assume that \eqref{e:lbd00} is satisfied for the type I term, we conclude that there exists $L\leq N^{2/3}$ and $N'\leq N$ such that for $L \exp(-(\log N)^{1/20})$ many $\ell\in [L/2,L]$ we have 
	\begin{align}\label{e:type1weyl}
		|\Expectation_{n<N'/\ell} \tilde{F}(g(\ell n)\Gamma)| \gtrsim  \frac{N}{N'} \exp(-(\log \delta^{-1})^{5C_8}) \gtrsim \exp(-(\log \delta^{-1})^{5C_8}) .
	\end{align}

    If we assume that \eqref{e:lbd00} is satisfied for the twisted type I term, then there exists $L\leq N^{2/3}$ and $N'\leq N$ such that for $L \exp(-(\log N)^{1/20})$ many $\ell\in [L/2,L]$ we have 
	\begin{align}\label{e:ttype1weyl}
		|\Expectation_{n<N'/\ell} \tilde{F}(g(\ell n)\Gamma)\chi_{q_{\textup{Siegel}}}(n)| \gtrsim  \exp(-(\log \delta^{-1})^{5C_8})  .
	\end{align}

	Finally, if we assume that \eqref{e:lbd00} is satisfied for the type II sum, there exists $N^{1/3}\leq L\leq N^{2/3}$  and $M\in [N\exp(-(\log \delta^{-1})^{3C_8})/L, N/L]$ such that 
	\begin{align}\label{e:type2weyl} \quad
		|\Expectation_{\ell,\ell'\in [L,2L]} \Expectation_{m,m'\in [M,2M]} \tilde{F}(g(m\ell)\Gamma)\overline{\tilde{F}(g(m\ell')\Gamma)}\overline{\tilde{F}(g(m'\ell)\Gamma)}\tilde{F}(g(m'\ell')\Gamma)| \gtrsim  \exp(-(\log \delta^{-1})^{5C_8}).
	\end{align}
	
	\textit{\textbf{Step 6:} Applying nilsequence Weyl sum technology for type I, twisted type I, and type II terms.} Now that we know at least one of \eqref{e:type1weyl}, \eqref{e:ttype1weyl}, or \eqref{e:type2weyl} must occur, we apply Weyl inverse theorem. Whichever is the case, the outcome is similar: see \cite{leng2023efficient}*{Lemma 5.2} for \eqref{e:type1weyl}, \cite{leng2023efficient}*{Lemma 5.3} for \eqref{e:ttype1weyl}, and \cite{leng2023efficient}*{Lemma 5.4} for \eqref{e:type2weyl}; we conclude that $g(n)$ has the following form
	\begin{align}
		g(n) = \epsilon(n)g_1(n)\gamma(n) 
	\end{align}
	where 
	\begin{itemize}
		\item There exists $C_{10}\geq 30C_8$ and  $1 \ll C_9 = O(1)$;
		\item $g_1$ takes values in a $\exp((\log \delta^{-1})^{5C_8})-$rational subgroup of $G$, call it $H$;
		\item The step of $H$ is at most $1$;
		\item $\gamma$ is $\exp((\log \delta^{-1})^{C_{10}})-$rational;
		\item $\epsilon$ is $(\exp((\log \delta^{-1})^{C_9}), N^{-1})$-smooth, which means that for all $n\in [N]$,
		\begin{align}
			&d(\epsilon(0),e_G) \lesssim  \exp( (\log \delta^{-1})^{C_9});\\
			&d(\epsilon(n),\epsilon(n-1)) \lesssim  \exp((\log \delta^{-1})^{C_9}) N^{-1}.
		\end{align}
	\end{itemize}
	If we substitute this into \eqref{e:beforetypes}, we may extract a (large) constant $C_{11}\geq  C_9,C_{10}$ so that 
	\begin{align}
		|\Expectation_{n<N} (\Lambda-\Lambda'_{\leq Q}) (n) \tilde{F}(\epsilon(n)g_1(n)\gamma(n)\Gamma)| > \exp(-(\log \delta^{-1})^{C_{11}}).
	\end{align}
	
	\bigskip
	
	\textit{\textbf{Step 7:} Pigeonholing.} Let $\Delta\leq \exp((\log \delta^{-1})^{2C_{10}})$ be the period of $\gamma$, so that $\gamma$ becomes constant on each residue class modulo $\Delta$. By the triangle inequality, 
	\begin{align}
		\Expectation_{u\in [\Delta]}
		\Bigl|\Expectation_{ \substack{n<N \\ n\equiv u\pmod\Delta}} \textbf{}(\Lambda-\Lambda'_{\leq Q}) (n) \bar{F}(\epsilon(n)g_1(n)\gamma(u)\Gamma)
		\Bigr|  > \exp(-(\log \delta^{-1})^{C_{11}}).
	\end{align}
	Applying the pigeonhole principle, we 
	extract  an arithmetic progression $P=\Delta[N/\Delta]+u_0$ with size $N \exp(-(\log \delta^{-1})^{C_{10}})$ and $\gamma_0=\gamma(u_0)$ such that 
	\begin{align}
		\Bigl|\Expectation_{n\in P} (\Lambda-\Lambda'_{\leq Q}) (n)\bar{F}(\epsilon(n)g_1(n)\gamma_0\Gamma)
		\Bigr| > \exp(-(\log \delta^{-1})^{C_{11}}) . 
	\end{align}
	Having held $\gamma $ constant,  we  use the 
	smoothness of $\epsilon$ to pass to a setting in which it is constant as well. 
	Decompose $P$ into $\{ P_m \}$, where each $P_m$ is the intersection of $P$ with consecutive intervals of length $|P|\exp(-(\log \delta^{-1})^{3C_9})$, so that the size of each $P_m$ is at most $|P|\exp(-(\log \delta^{-1})^{3C_9})$, and apply the triangle inequality to get
	\begin{align}
		\Expectation_{m\in [\exp((\log \delta^{-1})^{3C_9})]}|\Expectation_{n\in P_m} (\Lambda-\Lambda'_{\leq Q}) (n)\bar{F}(\epsilon(n)g_1(n)\gamma_0\Gamma)| > \exp(-(\log \delta^{-1})^{C_{11}}) . 
	\end{align}
	Using the pigeonhole principle, there exists $m_0< \exp((\log \delta^{-1})^{3C_9})$ and $P_0\coloneqq P_{m_0}$ such that 
	\begin{align}
		|\Expectation_{n\in P_0} (\Lambda-\Lambda'_{\leq Q}) (n)\bar{F}(\epsilon(n)g_1(n)\gamma_0\Gamma)| > \exp(-(\log \delta^{-1})^{C_{11}}).
	\end{align}
	We know that for every $x,y\in P_0$ we have 
	\begin{align}
		d(\epsilon(x),\epsilon(y)) \lesssim \sum_{n_i\in P_{0}}d(\epsilon(n_i),\epsilon(n_{i-1})) \lesssim |P_0| N^{-1} \exp((\log \delta^{-1})^{C_9}) \lesssim \exp(-(\log \delta^{-1})^{C_9}),
	\end{align}
	so we can assume that $\epsilon$ is also a constant on $P_0$ with size at most $N \exp(-(\log \delta^{-1})^{2C_{10}}-(\log \delta^{-1})^{3C_9})$. Now, choose $C_{12}$ to be larger than $10(C_{11}+C_9)$; the pigeonhole principle implies that there exists a sub-progression $P_1\subseteq P_0$ of length at most $N \exp(-(\log \delta^{-1})^{C_{12}})$ and $\epsilon_0\in G$ such that 
	\begin{align}\label{e:pigeonhole}
		|\Expectation_{n\in P_1} (\Lambda-\Lambda'_{\leq Q}) (n)\bar{F}(\epsilon_0g_1(n)\gamma_0\Gamma)| > \exp(-(\log \delta^{-1})^{C_{12}}).
	\end{align}
	
	\bigskip
	
	\textit{\textbf{Step 8:} Concluding the argument.}
	By \cite{leng2023efficient}*{Lemma B.6}, we can represent $\gamma_0=\{\gamma_0\}\gamma_1$ where $\gamma_1\in \Gamma$ and $\psi(\{\gamma_0\}) \in [-\frac{1}{2},\frac{1}{2})^d$, where $\psi$ is the Mal'cev coordinate map, see Definition \ref{d:nilmalcevdef}. Writing $g_2(n) = \{\gamma_0\}^{-1}g_1(n)\{\gamma_0\}$, we have 
	$$\epsilon_0 g_1(n)\gamma_0 \Gamma= \epsilon_0\{\gamma_0\} \left(\{\gamma_0\}^{-1}g_1(n)\{\gamma_0\}\right) \gamma_1 \Gamma = \epsilon_0\{\gamma_0\} g_2(n)  \Gamma.$$
	Define $F_1(t)= \bar{F}(\epsilon_0 \{\gamma_0\}^{-1} t)$. Since $\psi(\{\gamma_0\})$ and $d(\epsilon_0,e_G) \lesssim \exp(-(\log \delta^{-1})^{2C_{12}})$,  the Lipschitz norm of $F_1$ is at most
	\begin{align}
		\lesssim \exp((\log \delta^{-1})^{10C_9}+(\log \delta^{-1})^{C_6}) \lesssim \exp((\log \delta^{-1})^{C_{12}}).
	\end{align}
	Hence, \eqref{e:pigeonhole} reduces to  
	\begin{align}
		|\Expectation_{n\in P} (\Lambda-\Lambda'_{\leq Q}) F_1(g_2(n)\Gamma)| \gtrsim  \exp(-(\log \delta^{-1})^{C_{12}}) . 
	\end{align}
	Note that $g_2$ still takes values  inside a subgroup of $G$ with step at most $1$ and $F_1$ is defined in the corresponding step-$1$ nilmanifold. Applying the above argument possibly once more, we end up with a progression $P_2$ with size $N e^{-(\log \delta^{-1})^{C_{13}}}$ for some $C_{13}\geq 10^4C_{12}$  so that
	\begin{align}
		|\Expectation_{n\in P_2} (\Lambda-\Lambda'_{\leq Q})(n)| \gtrsim  \exp(-(\log \delta^{-1})^{C_{13}}).
	\end{align}
	Recall that $\delta = \exp(- (\log N)^\mathbf{c}) $ where we at last specialize $\mathbf{c}$ to be so small that 
	\begin{align}
		\exp(-(\log \delta^{-1})^{C_{13}}) =\exp(-(\log N)^{\mathbf{c}C_{13}}) \gtrsim  \exp(-(\log N)^{-1/20}),
	\end{align}
	yielding \eqref{e:overprog}; the proof is complete.   
\end{proof}

\bigskip 

{We are now ready for the proof of Proposition \ref{p:logsavings}.
With Lemma \ref{l:siegelcomp} in mind, we  need only to show that $\Lambda_{ \leq Q}$ and $\Lambda'_{ \leq Q}$ are close in $U^3$ norm.   
But, these two differ only by the Siegel zero component, if it exists.  Namely,  we show that  
\begin{align} \label{e;SiegelToShow}
	\|\Lambda_{ \leq Q}(n) \cdot n^{\sigma-1}\chi_{q_{\textup{Siegel} }}(n)\|_{U^3([N])} \lesssim N^{\frac{\sigma-1}{2}} q_{\textup{Siegel}}^{-1/18} .
\end{align}
The location of the Siegel zero  $\sigma =\sigma _Q $ can be very close to $1$.  And so, the principal subcase is when $\sigma =1$, and we address it in the following Lemma.   }

\begin{lemma}\label{l:U3siegel}
	\label{l;sigma=1} We have the estimate below, valid for $M\geq Q^{100} $ and $\epsilon >0$. 
	\begin{align} \label{e;sigma=1}
		\|\Lambda_{ \leq Q}(n)\chi_{q_{\textup{Siegel}}}(n)\|_{U^3[M]} \lesssim q_{\textup{Siegel}}^{-1/16+\epsilon}.
	\end{align}
\end{lemma}

\begin{proof}
	We in fact prove a slightly stronger conclusion; namely, expanding the  Heath-Brown model and applying the triangle inequality  we establish that for each
	$\epsilon >0$ 
	\begin{align}
		\label{e;;U3SigelQ-1}
		\sum_{\substack{(q_{\omega})\in [Q]^{8} } }
		\prod _{\omega }
		\frac{\mu(q_{\omega}) ^2 } {\phi(q_{\omega})}
		\Bigl\lvert 
		\Expectation_{(n,\mathbf{h})\in [M]^{4}}   
		\prod_{\omega\in \{0,1\}^3} 
		\chi_{q_{\textup{Siegel} }}(n+\omega\cdot\mathbf{h})
		c_{q_{\omega}}(n+\omega\cdot \mathbf{h})  
		\Bigr\rvert  \lesssim q_{\textup{Siegel}}^{-1/16+\epsilon}.
	\end{align}
	
	We capture a simplification in the 
	expectation above. 
	Fix $(q_{\omega})\in [Q]^{8}$, and  set  
	\[ Q_1 = q_{\textup{Siegel}}  \operatorname{lcm}(q_\omega ) =:q_{\textup{Siegel}} Q_2,\] which is a period for the function inside the expectation, which is much smaller than $M \geq Q^{100}$.
	The expectation in question is 
	\begin{align}  \label{errortermonleft}
		\Expectation_{(n,\mathbf{h})\in [M]^{4}} 
		&  
		\prod_{\omega\in \{0,1\}^3} 
		\chi_{q_{\textup{Siegel} }} 
		(n+\omega\cdot\mathbf{h})
		c_{q_{\omega}}(n+\omega\cdot \mathbf{h}) 
		\\&= 
		\Expectation_{(n,\mathbf{h})\in [Q_1]^{4}} 
		\prod_{\omega\in \{0,1\}^3} 
		\chi_{q_{\textup{Siegel} }} 
		(n+\omega\cdot\mathbf{h})
		c_{q_{\omega}}(n+\omega\cdot \mathbf{h}) +O(Q_1/M);
		\\ 
		\intertext{setting $q _{\omega } = q'_\omega \cdot q''_\omega $, where $q'_\omega = (q_\omega , q_{\textup{Siegel}})$, 
			and using periodicity}
		\\ \eqref{errortermonleft} & = 
		\Expectation_{(n,\mathbf{h})\in [q_{\textup{Siegel}}]^{4}} 
		\Expectation_{ (m , \mathbf{k}) \in [Q_2]^{4} }  
		\prod_{\omega\in \{0,1\}^3} 
		\chi_{q_{\textup{Siegel} }} (n+\omega\cdot \mathbf{h}) 
		c _{ q _\omega'}(n+\omega\cdot \mathbf{h})\\ 
		& \quad \times \prod_{\omega\in \{0,1\}^3} 
		c_{q_{\omega}''}(n+mq_{\textup{Siegel}}+\omega\cdot q_{\textup{Siegel}}\mathbf{k} + \omega\cdot \mathbf{h}) +O(q_{\textup{Siegel}}^{-20}).
		\\ 
		\intertext{Noting that 
			$\prod_{\omega\in \{0,1\}^3} c _{ q _\omega'}(n+\omega\cdot \mathbf{h})$ is identically 
			$\prod_{\omega\in \{0,1\}^3} \mu (q _\omega')$ 
			on the support of 
			$\chi_{q_{\textup{Siegel} }} (n+\omega\cdot \mathbf{h}) $,} 
		\eqref{errortermonleft} & =
		\pm\Expectation_{(n,\mathbf{h})\in [q_{\textup{Siegel}}]^{4}}
		\prod_{\omega\in \{0,1\}^3} 
		\chi_{q_{\textup{Siegel} }}(n+\omega\cdot \mathbf{h}) \\
		& \quad \times 
		\Expectation_{(m,\mathbf{k})\in [Q_2]^{4}}  
		\prod_{\omega\in \{0,1\}^3} 
		c_{q_{\omega}''}(n+mq_{\textup{Siegel}}+\omega\cdot q_{\textup{Siegel}}\mathbf{k} + \omega\cdot \mathbf{h}).
		\\
		\intertext{Since $q''_\omega $ and $q_{\textup{Siegel}}$  are coprime for all $\omega \in \left\{0,1\right\}^3$, the expectation over $(m,\mathbf{k})$ can be simplified, which gives the total contribution}   
		\eqref{errortermonleft} & = 
		\pm\Expectation_{(n,\mathbf{h})\in [q_{\textup{Siegel}}]^{4}}
		\prod_{\omega\in \{0,1\}^3} 
		\chi_{q_{\textup{Siegel} }}(n+\omega\cdot \mathbf{h}) 
		\times \mathbf{E}(q _{\omega }''), 
		\\
		\intertext{where we define}
		\mathbf{E}(q _{\omega }'') 
		& \coloneqq \Expectation_{(m,\mathbf{k})\in [Q_2]^{4}}  
		\prod_{\omega\in \{0,1\}^3} 
		c_{q_{\omega}''}(m +\omega\cdot \mathbf{k}).
	\end{align}
	Here, we note that the error term on the left at the beginning of these estimates 
	\eqref{errortermonleft} is small enough, due to our lower bound on  $M$. 
	And, in the last line, the expectation in $ (m,\mathbf{k}) $ 
	is over several full periods. 
	
	With the estimate from \cite{MR4875606}*{Lemma 5.6}  of the $U^3$ norm of the Siegel character, namely, 
	\begin{align}
		\lVert \chi_{q_{\textup{Siegel} }} \rVert _{U^3[Q_1]} ^{8} 
		\ll q_{\textup{Siegel}}^{-1/2+o(1)}  
	\end{align}
	we are left to prove the estimate 
	\begin{align} \label{e;SigelToProve}
		\sum_{ \substack{ (q_\omega ) \in [Q] ^{8}}}
		\prod_{\omega\in \{0,1\}^3} 
		\frac{\mu(q_{\omega}) ^2 } {\phi(q_{\omega})}
		\mathbf{E}(q _{\omega }'')  \lesssim  
		q_{\textup{Siegel}}^{o(1)} . 
	\end{align}
	
	\smallskip 
	
	The term $\mathbf{E}( q _{\omega }'')$ is non-negative and   is estimated in Lemma \ref{l;Ep} below. So,   
	we bound the left side of \eqref{e;SigelToProve} by 
	the following expression. 
	Collect
    $D = \{d \colon d \mid q_{\textup{Siegel}} \},$
	and for $(q_\omega'' )\in [Q] ^{8}$, set  
	\[ R = R(q_\omega'') =   \prod_{\omega\in \{0,1\}^3} q_\omega'' .\]
	By Lemma \ref{l;Ep}, the left side of \eqref{e;SigelToProve} is at most  
	\begin{align} \label{e;dejavu}
		\begin{split}
			A \times B\coloneqq &\sum_{ (d_\omega ) \in D ^{8} } 
			\prod_{\omega\in \{0,1\}^3} \phi (d_\omega ) ^{-1} 
			\\ & \qquad\times 
			\sum_{ \substack{ (q_\omega'' ) \in [Q] ^{8} \\ 
					\textup{Rad}( R ) ^{4} \mid  \prod_\omega  q_\omega''  }} 
			\prod_{\omega\in \{0,1\}^3}  \mu (q_w'') ^2 
			\prod_{\substack{ p\mid R  } }\phi (p) ^{-3}  .   
		\end{split}
	\end{align}  
	
	The first term is straightforward to estimate: with $\sigma (n)$ being the sum of divisors function, it is at most
	\begin{align}
		A  \lesssim  \Bigl[ \frac {\sigma (q_{\textup{Siegel}})}
		{\phi(q_{\textup{Siegel}})} \Bigr] ^{8} = q_{\textup{Siegel}} ^{o(1)}. 
	\end{align} 
	The second term is akin to \eqref{e;dejavu}.  
	In the definition of $B$, we certainly have 
	$R \mid \textup{Rad}( R  ) ^{8}$, so   
	we bound 
	\begin{align}
		\sum_{\substack{R \in \mathbb{N} \\  \textup{Rad}(R)^4 | R | \textup{Rad}(R)^8} } 
		\phi ( \textup{Rad}(R) )^{-3} 
		& \lesssim  
		\sum_{q \in \NN : q \text{ is squarefree}} 
		5^{\omega(q)} \phi(q)^{-3} \lesssim 1, 
	\end{align}
	which follows from standard pointwise estimates on functions $\omega$ and $\phi$. This establishes \eqref{e;SigelToProve}, thereby completing the proof of Lemma \ref{l:U3siegel}. 
\end{proof}

It remains to estimate $\mathbf{E}(q _{\omega }'')$; we accomplish this below.

\begin{lemma}  \label{l;Ep}
	Let $ (q_ {\omega} )_{ \omega \in \{ 0, 1\}^3 } $ be square-free integers. Set $R=\prod_{\omega \in \left\{0,1\right\}^3}q_{\omega}$ and  $Q$ to be a large integer such that $  \operatorname{lcm}{(q_\omega ) } \mid Q $.  Then
	\begin{align} \label{expectation-of-ramanujan-products}
		0\leq 
		\Expectation_{ (m,\mathbf{j}) \in [Q] ^{4} } 
		\prod _{\omega \in \{ 0, 1\}^3 } c_{q_\omega }(m+ \omega \cdot \mathbf{j})
		\leq 
		\begin{cases}
			0  &  \textup{Rad}(R) ^{4} \nmid R 
			\\ 
			\prod _{p\mid R} (p-1)^{v_p(R)-3} & \textup{otherwise}
		\end{cases}
	\end{align}
	where $\textup{Rad}(M)$ is the radical of integer $M$ and  $v_p(M)$ is the largest power of $p$ that divides $M$. 
\end{lemma}

The argument below reduces to a counting argument very similar to Lemma \ref{l:algorithm}, with the only difference being certain $\pm 1$ sign choice, which are inessential; we just provide the reduction, and leave the details to the interested reader.

\begin{proof}
Each term in the expectation is $\operatorname{lcm}{(q_{\omega})}$ periodic, so, without loss of generality we may assume that $Q= \operatorname{lcm}{(q_{\omega})}$.	Using the multiplicative property of Ramanujan sums and the fact that $q_{\omega}$ are square-free  for $\omega \in \{0, 1 \}^3$ we have that \eqref{expectation-of-ramanujan-products} is given by
	\begin{align}
		\Expectation_{ (m,\mathbf{j}) \in [Q] ^{4} } 
	\prod _{\omega \in \{ 0, 1\}^3 } c_{q_\omega }(m+ \omega \cdot \mathbf{j})=	\Expectation_{ (m,\mathbf{j}) \in [Q] ^{4} } 
		\prod _{p\mid R} 
		\prod _{\substack{\omega \in \{0, 1 \}^3 \\ p \mid q_\omega } }
		c_{p }(m+ \omega \cdot \mathbf{j}). 
	\end{align}
	Observe that since $Q$ is square free, we may decompose $\Z_{Q} \equiv \bigoplus_{p \mid Q} \Z_{p} $ which implies that
	\begin{align}
		\Expectation_{ (m,\mathbf{j}) \in [Q] ^{4} } 
		\prod _{p\mid R} 
		\prod _{\substack{\omega \in \{0, 1 \}^3 \\ p \mid q_\omega } }
		c_{p }(m+ \omega \cdot \mathbf{j})	= 
		\prod _{p\mid R}
		\Expectation_{ (m,\mathbf{j}) \in [p] ^{4} } 
		\prod _{\substack{\omega \in \{0, 1 \}^3 \\ p \mid q_\omega } }
		c_{p }(m+ \omega \cdot \mathbf{j})
	\end{align}

	Expand each $c_p(v + \omega\cdot \mathbf{k}) = \sum_{1\leq a_\omega<p } e(a_\omega (v + \omega\cdot \mathbf{k})/p) $, so that   
	\begin{align}
		\Expectation_{(m,\mathbf{k})\in [p]^{4}}
		\prod_{\substack{\omega\in  \{0, 1\}^3 \\ p \mid q_{\omega} } }
		c_{p}(m + \omega\cdot \mathbf{k}) 
		& =  \Expectation_{(m,\mathbf{k})\in [p]^{4}}
		\prod_{\substack{\omega\in  \{0, 1\}^3 \\ p \mid q_{\omega} } }
		\sum _{ 1\leq a_\omega  <p} e( a_ \omega (v + \omega\cdot \mathbf{k})/p)
		\\ 
		& =
		\Expectation_{(m,\mathbf{k})\in [p]^{4}}
		\sum_{1 \leq a_{\omega} <p }
		e\Bigl( \sum_{\substack{\omega\in  \{0, 1\}^3 \\ p |q_\omega  } }
		a_\omega (m + \omega\cdot \mathbf{k})/p\Bigr) 
		\\   \label{e;4rows}
		& = 
	\sum_{ \omega \in \left\{0,1\right\}^3, \;  p \mid q_{\omega} }		\sum_{ 1\leq a_\omega <p  }
		\mathbf{1}_{ p\mid \sum_{\omega \in \{0, 1\}^3  } a_\omega  } 
		\times 
		\prod _{j=1}^3 
		\mathbf{1}_{ p\mid \sum_{\substack{\omega \in \{0, 1\}^3  
					\\ \omega _j =1  } } a_\omega.  } 
	\end{align}

\end{proof}

To complete the proof of  Proposition \ref{p:logsavings}, 
it remains to prove \eqref{e;SiegelToShow}.  This in turn follows from this Lemma and the elementary bound
\begin{align}
    \| \Lambda_{\leq Q} \|_{U^3([\sqrt{N}])} \lesssim \| \Lambda_{\leq Q} \|_{L^2([\sqrt{N}])}.
\end{align}

	\begin{lemma} For all $A \geq 1$
	 \[ \left \| \Lambda_{ \leq Q}(n) \left(n^{\sigma-1} \chi_{q_{\textup{Siegel}}}(n)\right) \cic{1}_{\left(\sqrt{N},N\right]} \right \|_{U^3[N]} \lesssim N^{\frac{\sigma-1}{2}} q_{\textup{Siegel}}^{-1/16+\epsilon}  \lesssim_{A} (\log N)^{-A}  \]
	\end{lemma}
	\begin{proof}
		We write \[ n^{\sigma-1} \cic{1}_{\left(\sqrt{N},N\right]}(n)=\int_{\sqrt{N}}^N  (1-\sigma) M^{\sigma-2} \cic{1}_{[M]}(n) \d M+N^{\sigma-1} \cic{1}_{[N]}(n)-(\sqrt{N})^{\sigma-1} \cic{1}_{[\sqrt{N}]}(n);  \] due to the subadditivity of $U^3[N]$ and Lemma \ref{l;sigma=1} it suffices to prove the estimate \[ \left \| \int_{\sqrt{N}}^{N} \Lambda_{ \leq Q}(n) M^{\sigma-2} \cic{1}_{[M]}(n) \chi_{q_{\textup{Siegel}}}(n) \d M \right \|_{U^3[N]}^8 \lesssim N^{4(\sigma-1)}  q_{\textup{Siegel}}^{-1/2+8\epsilon}. \] 
		Equivalently, we need to prove that \[  \frac{1}{N^4}\sum_{x,h_1,h_2,h_3 \in \Z}  \Delta_{h_1,h_2,h_3} \left( \int_{\sqrt{N}}^{N} \Lambda_{ \leq Q}(n) M^{\sigma-2} \cic{1}_{[M]}(n) \chi_{q_{\textup{Siegel}}}(n)  \d M  \right)(x)  \lesssim N^{4(\sigma-1)} q_{\textup{Siegel}}^{-1/2+8\epsilon};    \] the LHS of the aforementioned estimate is equal to   \begin{equation} \label{averagedgowers}
			\begin{split}
		&	\qquad  \frac{1}{N^4} \sum_{(x,\mathbf{h}) \in \Z^4} \int_{(\sqrt{N},N]^8} \prod_{\omega \in \left\{0,1\right\}^3}M_{\omega}^{\sigma-2} C^{|\omega|} \left(\Lambda_{ \leq Q}(x+ \omega \cdot \mathbf{h})\chi_{q_{\textup{Siegel}}}(x+ \omega \cdot \mathbf{h}) \cic{1}_{[M_{\omega}]}(x+\omega \cdot \mathbf{h})\right)   \d \mathcal{M} \\ & = \frac{1}{N^4} \int_{(\sqrt{N},N]^8} \prod_{\omega \in \left\{0,1\right\}^3}M_{\omega}^{\sigma-2}   \sum_{(x,\mathbf{h}) \in \Z^4} \prod_{\omega \in \left\{0,1\right\}^3} C^{|\omega|}\left(\Lambda_{ \leq Q}(x+ \omega \cdot \mathbf{h}) \chi_{q_{\textup{Siegel}}}(x+ \omega \cdot \mathbf{h}) \cic{1}_{[M_{\omega}]}(x+\omega \cdot \mathbf{h})\right)  \d \mathcal{M}.
		\end{split}
		\end{equation} 
        Where $\d \mathcal{M}=\bigotimes_{\omega \in \left\{0,1\right\}^3} \d M_{\omega}. $
		At this point we are able to use the Gowers-Cauchy-Schwarz Inequality \ref{CauchySchwarzGowers} to obtain that  for each fixed $\left\{M_{\omega}\right\}^{\left \{0,1\right \}^3} \in  \left(\sqrt{N},N\right]^8 $ we have that \[ \begin{split}
			 \sum_{(x,\mathbf{h}) \in \Z^4} \prod_{\omega \in \left\{0,1\right\}^3} C^{|\omega|}\left(\Lambda_{ \leq Q}(x+ \omega \cdot \mathbf{h}) \cic{1}_{[M_{\omega}]}(x+\omega \cdot \mathbf{h})\right)  &\leq  \prod_{\omega \in \left\{0,1\right\}^3} \left \| \Lambda_{ \leq Q} \chi_{q_{\textup{Siegel}}} \cic{1}_{[M_{\omega}]}  \right \|_{U^3(\Z)} \\ & \lesssim \prod_{\omega \in \left\{0,1\right\}^3}  M_{\omega}^{1/2} \left \| \Lambda_{ \leq Q} \chi_{q_{\textup{Siegel}}}  \right \|_{U^3([M_{\omega}])} \\ & \lesssim q_{\textup{Siegel}}^{-8/16+8\epsilon}   \prod_{\omega \in \left\{0,1\right\}^3}  M_{\omega}^{1/2}. 
		\end{split} \] 
Using this estimate and the fact that $M_{\omega} \geq Q^{100}$ we see that \eqref{averagedgowers} is controlled by \[\begin{split}
	 \frac{q_{\textup{Siegel}}^{-1/2+8\epsilon}}{N^4}\left( \int_{(\sqrt{N},N]} M^{\sigma-3/2}  \d M\right)^{8} &  = \frac{q_{\textup{Siegel}}^{-1/2+8\epsilon}}{N^4} \left( \frac{1}{\sigma-\frac{1}{2}} \right)^8  \left( N^{\sigma-1/2}-(\sqrt{N})^{\sigma-1/2} \right)^8 \\ & \lesssim N^{8(\sigma-1)}  q_{\textup{Siegel}}^{-1/2+8\epsilon}.
\end{split} \]  To establish our desired bound, 	 \[ \left \| \Lambda_{ \leq Q}(n) \left(n^{\sigma-1} \chi_{q_{\textup{Siegel}}}(n)\right) \cic{1}_{\left(\sqrt{N},N\right]} \right \|_{U^3[N]} \lesssim  N^{\frac{1}{2}(\sigma-1)} q_{\textup{Siegel}}^{-1/16+\epsilon},   \] we separate according to whether $q_{\textup{Siegel}} \geq (\log N)^{32A}$, in which case the bound  clearly presents. But, when $q_{\textup{Siegel}} \leq (\log N)^{32A}$ we have that \[N^{\frac{\sigma-1}{2} }  \leq (N^{-C_{A}/q_{\textup{Siegel}}^{1/64A}})^{1/2} \leq  \exp\left (- \frac{C_A}{2} \left (\log N \right )^{1/2}\right )\] from which the desired bound is achieved.
	\end{proof} 
	
Having developed our requisite additive-combinatorial machinery, we can quickly prove our main results.

\section{Proof of the Wiener-Wintner Theorem}

We phrased our Wiener-Wintner result as a uniform 
result over all continuous functions $\phi \colon \mathbb{R}/\mathbb Z \to \mathbb C$. 
By the Weierstrass approximation theorem, 
it suffices to prove the Wiener-Wintner Theorem for exponential functions. That is, we show that 
there is set $X_f\subset X$ of full measure, so that for each $x\in X_f$,  
\begin{align}  
  \lim_{N\to\infty}   
  \Expectation_{[N]} \Lambda(n) e^{2\pi i kn \theta} f(T^{n} x) 
   \qquad\textup{exists for all $k\in \mathbb{N}$.}  
\end{align}

If $f$ is an eigenfunction of $T$, the result follows from Vinogradov's Theorem on the pointwise convergence of exponential prime averages.  Due to the Prime Ergodic Theorem \cite{MR995574}, we see that the same result holds for any function in the $L^2$ closure of eigenfunctions of $T$. 

Thus, we can assume that $f \in L^{\infty}(X)$  is weakly mixing; we will show that in this case
\begin{align}
    \limsup_N \sup_\theta |\Expectation_{[N]} \Lambda(n) e^{2\pi i kn \theta} f(T^{n} x)| = 0
\end{align}
almost surely. Seeking a contradiction, suppose there existed some $c' > 0$ so that
\begin{align}
    \int_X    \limsup_N \sup_\theta |\Expectation_{[N]} \Lambda(n) e^{2\pi i kn \theta} f(T^{n} x)| \ d\mu(x) > c';
\end{align}
by the ergodic decomposition, there is no loss of generality in assuming that $(X,\mu,T)$ is ergodic.

The crucial observation is as follows: since $f$ is weakly mixing, the  uniform Wiener-Wintner Theorem holds, 
as given in  \eqref{e;uniformWW}; but, since the average in \eqref{e;weaklyMixing} is over positive quantities, we see that 
if $f$ is weakly mixing with respect to $T$, it is weakly mixing with respect to $T^k$ for any integer $k$.  
Accordingly, we can fix a single full measure subset 
$X_f\subset X$ so that for all $x\in X_f$  we have 
\begin{align}  
  \lim_N  
  \sup_\theta \Bigl\lvert 
  \Expectation_{[N]} e^{2\pi i n \theta} f(T^{kn} x)
  \Bigr\rvert =0,  \qquad k\in \mathbb{N}. 
\end{align}

With this understanding, we set up a transference argument.  
For integers $K$ and $N_0$, and $0<\kappa < 1/4$,  consider the set 
\begin{align}
    X_{\kappa} \coloneqq  \bigl\{  x \in X \colon  \max_{ 1\leq  k \leq K}\sup_{\theta} \sup_{N \geq N_0}  |\mathbb{E}_{[N]} f(T^{ k n} x) e(n\theta)| \leq \kappa\bigr\}.
\end{align}
We can choose $N_0$ so large that $\mu (X_{ \kappa}) \geq 1- \kappa$.  
And then, by the Pointwise Ergodic Theorem, we can fix $J$ so large that  $\mu(X'_\kappa) > 1- 2\kappa$, where 
\begin{equation}
  X'_{\kappa} \coloneqq \Bigl\{  x\in X \colon     \Expectation_{j \in [J]} T^j \mathbf{1}_{X_{\kappa}} (x)  > 1- 2\kappa 
  \Bigr\}. 
\end{equation}
Then, as $f$ is a bounded function, for $x\in X'_\kappa$,  we certainly have 
\begin{equation}
   \Expectation_{j \in [J]}
    \max_{ 1\leq  k \leq K}\sup_{\theta} \sup_{N \geq N_0}  |\mathbb{E}_{[N]} f(T^{ k n +j} x) e(n\theta)|^4 < 10\kappa. 
\end{equation}
Above, we have imposed an $L^4$ norm. 
From the Lemma below, after appropriately choosing $\kappa$, $K$, $N_0$ and $J$,  we conclude that again with $x\in X'_\kappa$, 
\begin{equation}
   \Expectation_{j \in [J]}
  \sup_{\theta} \sup_{N_0 < N < J}  |\mathbb{E}_{[N]} \Lambda (n) f(T^{ k n +j} x) e(n\theta)|^4 \lesssim \kappa ^{1/4}. 
\end{equation}
We integrate this over $X$, and again appeal to the Pointwise Ergodic Theorem to see that the Prime Wiener-Wintner Theorem holds.


\begin{lemma}\label{l;WW}
Let $f \colon [-2J,2J]  \to \mathbb{C}$ be a $1$-bounded function. 
Further assume that for some $0<\epsilon <1$, 
and integers $K\approx  2 ^{ \epsilon^{-100} } < N_0 ^{1/5} <J$,  we have the inequality 
\begin{align} \label{e;UniformK}
  \Expectation_{[J]} 
  \sup_{\theta} \sup_{N > N_0}
\bigl\lvert \Expectation_{n\in [N]}  e(\theta n) f(x-k n)  \bigr\rvert  ^4   < \epsilon^4 . 
\end{align}
Then, the following estimate holds for the prime averages $A_N ^{\theta } $,   for all integers $J > N_0$,  
\begin{align} \label{e;WWconconclusion}
\Expectation_{x\in [J]} \max_{ 
 N_0\leq N  <  J} \sup_{\theta } \lvert A_N ^{\theta} f(x) \rvert ^4 \lesssim \epsilon. 
\end{align}
\end{lemma}

\begin{proof}
Observe that it suffices to control the maximum over integers $N = \lceil (1+ \epsilon^{1/3})^k\rceil$, for integers $k\in \mathbb{N}$.  This is because for  $N=\lceil (1+ \epsilon^{1/3})^k\rceil < N'< \lceil (1+ \epsilon^{1/3})^{k+1}\rceil$, we have 
\begin{align} \label{e;WW1}
    \lvert   \tfrac N{N'} A_N ^{\theta} f(x) 
    - A_{N'} ^{\theta} f(x)\rvert 
    & \leq 
    \frac{1}{N'} \sum_{ n= N+1}^{N'} \Lambda (n) \lesssim \epsilon^{1/3}. 
\end{align}
Above, $N > N_0 \simeq 2^{\epsilon^{-100}}$, and we only need the trivial bound $\Lambda(n)\leq \log n$.  
Below, all integers $N$ will be of the form $\lceil (1+ \epsilon^{1/3})^k\rceil$, but this is supressed in the notation. 

\smallskip 

We replace the von Mangoldt function with the Heath-Brown model.  
Define 
\[
 B_N^\theta g(x) \coloneqq  \Expectation_N g(x-n) e(n \theta) \Lambda_{ \leq N'}(n),  \qquad  N' \coloneqq  \exp( (\log N)^{1/10}), 
 \]
 where $\Lambda_{ \leq N'}$ is defined in \eqref{e:lambda000}. 
 Recall
 Proposition \ref{p:logsavings}, our estimate of the 
 $U^3$ distance between the von Mangoldt function and the Heath-Brown model.  
 And, the inequality \eqref{e;U3control}, which controls averages against weights with small $U^3$ norm.  
 Employing both, we have
\begin{align}  \label{e;WW2}
\sum_{ N_0 < N < J }
\Expectation_{[J]} \lvert (B_{N} ^{\theta } - A _{N} ^{\theta }) f \rvert ^{4}
& \lesssim 
\sum_{ N_0 < N < J } (\log N) ^{-40}  \lesssim  \epsilon^{-1/3} (\log N_0)^{-20}. 
\end{align}
We analyze the averages with respect to $B_N ^{\theta}$ below. 

\smallskip 
 With $\Lambda_Q$  as in \eqref{e:lambdaQ},  namely that part of the Heath-Brown model with rational denominators approximately $Q$, define
\begin{align}
    L_{Q,N}^\theta g(x) \coloneqq  \Expectation_N g(x-n) e(n \theta) \Lambda_{Q}(n). 
\end{align}
The Ramanujan sum $c_q$ is $q$ periodic. 
Thus, $ \Lambda_Q(n)$ is periodic with period  $ P_Q\coloneqq  \text{lcm}(q \approx Q) $, an essential property for us; recall the bounds
\[ 2^Q \leq P_Q \leq 3^Q.\]
There are two bounds we prove for the averages associated with these terms, according to the relative size of $3^Q$ and $K$:

In the first place, when $3^Q \leq K$, we appeal to the  uniform Wiener-Wintner property; in particular, our key uniformity hypothesis \eqref{e;UniformK} is in place. 
For $N > N_0 \gg K^5$, we express
 \begin{align}
  L_{Q,N}^\theta f (x) 
  & = \Expectation_{ m \in [P_Q]} \Lambda _Q(m) e(\theta m)
  \times 
  \Expectation_{ n\in [N/P_Q]} e(P_Q\theta n) f(x+m+nP_Q) 
  +O(P_Q/N) 
  \\& \ll  P_Q/N_0   + 
  \Expectation_{ m \in [P_Q]} \lvert \Lambda _Q  (m) \rvert   
   \times  \max_{ 1\leq k \leq K} \sup_{\theta} \sup_{N > N_0}
\bigl\lvert \Expectation_{n\in [N]}  e(\theta n) f(x-k n)  \bigr\rvert. 
 \end{align}
To bound the $L^1$ norm of $\Lambda_Q$,  we use the estimate \cite{KMTT}*{(4.7)}, which gives control of integer powers of $\Lambda_Q$. It is 
\begin{align}
    \label{e;LQk} 
    \Expectation_{[N]} \lvert \Lambda_Q (n)|^k \lesssim (1+\log Q)^{2^k+k} , \qquad k\in \mathbb{N}. 
\end{align}
We conclude that 
\begin{align}
   \sum_{Q \colon P_Q \leq K } \Expectation_{x\in [J]} \max_{ 
 N_0\leq N  <  J} \sup_{\theta } \lvert L_{Q,N}^\theta f(x) \rvert ^4 
 &\lesssim
 \sum_{Q \colon P_Q \leq K } P_Q/N_0   +  (1 +\log Q) ^3  \epsilon^4  \lesssim \epsilon . 
  \label{e;WW3} 
\end{align}
The sum is over $Q$ being powers of $2$, and it uses our assumption that $K\approx  2 ^{ \epsilon^{-100} } < N_0 ^{1/5}$.

\medskip

It remains to address $Q$ such that $3^Q > K \approx 2 ^{\epsilon^{-100} } $, that is $ Q \gg \epsilon^{-100}$. 
Our argument has two stages. First, we control a restricted maximum over $N$. 
Appealing to the $U^3$ control 
\eqref{e;U3control} and the fixed complexity estimate 
\eqref{e:LambdaQ}, we have 
\begin{align}
\Expectation_{x\in[ J]} \max_{\substack{  Q ^{20}\leq N  <  4^Q  }}
 \sup_{\theta } \lvert L_{Q,N}^\theta  (f)(x) \rvert ^4
 & \leq 
 \sum_{\substack{ Q ^{20}\leq N  <   4^Q  }}
 \Expectation_{x\in [J]}   
 \sup_{\theta } \lvert L_{Q,N}^\theta  f(x) \rvert ^4
 \\ 
 & \lesssim  
  \sum_{\substack{  Q ^{20}\leq N  <  4^Q }} Q^{-\frac{5}{4}}  
\\  & \lesssim  \epsilon^{-1/3} Q ^{-1/4} 
\\ & \lesssim \epsilon.   
\end{align}


In addition, our maximum over $N$ goes up to $4^Q$, well beyond the period of $\Lambda_Q$. 
With an additional argument, we then have a bound with  maximum over $N$ being as large as possible. Namely,  
\begin{equation}  \label{e;LQ<}
\Expectation_{x\in [J]}  \sup_ \theta  \max_{4^Q < N < J}
 \lvert L_{Q,N}^\theta f (x) \rvert ^{4} 
 \lesssim   \epsilon.  
\end{equation} 
The point is that we are forming averages with respect to a weight that is periodic, with period $P_Q$. 
  Let $\mathcal I$ be collection of pairwise disjoint intervals  of  length  $P_Q$. 
   Require that the union of the intervals in  $ \mathcal{I}$  covers $[J]$, and that $\mathcal{I}$ has  minimal cardinality. 
  Set $\Phi = \sum_{I\in \mathcal I} \phi_I$, where 
  \begin{equation}
      \phi_I =  
 \sup_{\theta } \bigl\lvert L_{Q,P_Q}^\theta  (f \mathbf 1_{ 3 I })\bigr\rvert
  \end{equation}
 We certainly have  
\begin{equation}  \label{e;maxUpToJ}
    \Expectation_{[J]} \lvert M' \Phi \rvert ^4 
    \lesssim \Expectation_{[J]} \lvert  \Phi \rvert ^4 \lesssim   \epsilon, 
\end{equation}
where $M'$ is the Hardy Littlewood maximal function, evaluated along a progression of step size $P_Q$. 
To be explicit, 
\begin{align}
    M' g(x) = \sup_{t \geq 1} t ^{-1}\sum_{s=0}^{t-1} \lvert g(x+tP_Q) \rvert . 
\end{align} 
The function $\Phi$ is introduced for this reason:  
For any point $x\in [J]$, and $N$ a multiple of $P_Q$, we have 
$\lvert L_{Q,N}^\theta f(x)\rvert \leq M' \Phi (x)$.   
But the averages  we need to control require  $N > 4^Q$, with $N$   of the form $\lceil (1+\epsilon^{1/3})^k \rceil$, for integers $k$.  We have 
\begin{align}
    \Expectation_{x\in [J]}  \sup_ \theta  \max_{ 4^Q< N < J}
 \lvert L_{Q,N}^\theta f (x) \rvert ^{4} 
 &\lesssim 
 \Expectation_{x\in [J]}  (M' \Phi)^4 + 
 (3/4)^Q\sup_{ K \colon \lvert K \rvert \leq P_Q} \Bigl[\Expectation_{K} \lvert \Lambda_Q (n)|^4
 \Bigr] ^{1/4} 
 \\
     &\lesssim \epsilon . 
\end{align}
The last line depends upon \eqref{e;LQk}.   
 This completes the proof.  
\end{proof}

\appendix

	\appendix
	
	\section{Approximation of the von Mangoldt Function in Arithmetic Progressions}

In this section we provide a  proof that $ \Lambda_{\leq Q}'$, 
defined in  \eqref{e:Lambdaleq}, imitates the behaviour of the von Mangoldt function on quasipolynomially dense arithmetic progressions:
throughout this section we regard our scale, $N$, as fixed.  We recall the definition 
\[ Q \coloneqq  Q_N \coloneqq  \exp((\log N)^{1/10})\] 
and $q_{\textup{Siegel}}\leq Q^{1/2}$, see \eqref{e;whereIsSigma} and definitions after that.  Specifically, we have the following lemma.
{\begin{lemma}\label{l:APest}
		Suppose $\mathcal{P} \subset [N]$ is an arithmetic progression with gap size $ q \leq Q^{\frac{1}{2}}.$
		Then 
		\begin{align}\label{e:APest} 
			|\sum_{n \in \mathcal{P}} \, \big( \Lambda - \Lambda_{\leq Q}'\big) (n)| \lesssim N \exp(-c (\log N)^{1/10});
		\end{align}
        compare to \eqref{e:overprog}.
\end{lemma}}

\begin{proof}
    Any arithmetic progression $\mathcal{P} \subset [N]$ can be expressed in the form 
    \[ \{ N'' < n \leq N' : n \equiv a \pmod q \}\] for some $a \in \{1, 2, \ldots, q \}$ and $0 < N'' \leq N' \leq N$. By the triangle inequality, it thus suffices to establish the bounds:
    \begin{align}
        \sum_{\substack{n \leq N' \\ n \equiv a \pmod q}} \Lambda(n) = \sum_{\substack{n \leq N' \\ n \equiv a \pmod q}} \Lambda_{\leq Q}'\big (n) + O(N \exp(-c \log^{1/10} N))
    \end{align}
    for any $a \in \{1, 2, \ldots, q\}$ and $0 < N' \leq N$.
    We first need to exclude the case when $q > \exp(\log^{1/9}N)$ so that Page theorem becomes non-trivial. But, in this case, the left-hand side of above equation is of order $O(N \exp(-\log^{1/9} N) \log N)$ by bounding $\Lambda(n)\leq \log N$ trivially. For the sum in the right hand side, we apply the Cauchy-Schwartz to conclude that 
    \begin{align}
        |\sum_{\substack{n \leq N' \\ n \equiv a \pmod q}} \Lambda_{\leq Q}'\big (n)| \lesssim \left(\frac{N'}{q}\right)^{1/2} \left(\sum_{\substack{n \leq N' }} |\Lambda_{\leq Q}\big (n)|^2\right)^{1/2}, 
    \end{align}
    where we used the pointwise bound $|\Lambda_{\leq Q}'|\leq 2 |\Lambda_{\leq Q}|$. Using \cite[(4.7)]{KMTT}, we have 
    \begin{align}
        |\sum_{\substack{n \leq N' \\ n \equiv a \pmod q}} \Lambda_{\leq Q}'\big (n)| \lesssim N' (\log Q)^{3} q^{-1/2} \lesssim N' \exp(- \frac{1}{3}(\log N)^{1/9}). 
    \end{align}
    So, it suffices to consider the case where $\mathcal{P}$ has a comparatively small gap size, $q \leq \exp(\log^{1/9}N)$. By  \cite[Theorem 5.27]{iwaniec2021analytic}
    \begin{align}
        \sum_{\substack{n \leq N' \\ n \equiv a \pmod q}} \Lambda(n) = \frac{N'}{\phi(q)} \Bigl(1 - \chi_{q_{\text{Siegel}}}(a) \mathbf{1}_{q_{\text{Siegel}} | q} \frac{(N')^{\sigma - 1}}{\sigma} \Bigr) \mathbf{1}_{(a, q) = 1} + O(N \exp(-c \log^{1/10}N)).
    \end{align}
    So, by the definition of $\Lambda_{\leq Q}'(n)$ and a rearrangement of the order of summation, it suffices to show that:
    \begin{align} \label{arithmetic-progression-sum-heath-brown}
        \sum_{\substack{n \leq N' \\ n \equiv a \pmod q}} \Lambda_{\leq Q}(n) = \frac{N'}{\phi(q)} \mathbf{1}_{(a, q) = 1} + O(N \exp(-c \log^{1/10} N))
    \end{align}
    and 
    \begin{align} \label{arithmetic-progression-sum-siegel}
        \sum_{\substack{n \leq N' \\ n \equiv a \pmod q}} \Lambda_{\leq Q}(n) n^{\sigma - 1} \chi_{q_{\text{Siegel}}}(n) = \frac{1}{\phi(q)} \chi_{q_{\text{Siegel}}}(a) \mathbf{1}_{q_{\text{Siegel}} | q} \frac{(N')^{\sigma}}{\sigma} \mathbf{1}_{(a, q) = 1} + O(N \exp(-c \log^{1/10}N)).
    \end{align}
    For \eqref{arithmetic-progression-sum-heath-brown}, we expand
    \begin{align}
        \sum_{\substack{n \leq N' \\ n \equiv a \pmod q}} \Lambda_{\leq Q}(n) &= \sum_{\substack{n \leq N' \\ n \equiv a \pmod q}} \sum_{t \leq Q} \frac{\mu(t)}{\phi(t)} \sum_{d  |(t, n)} \mu(\frac{t}{d}) d \\
        &= \sum_{t \leq Q} \frac{\mu(t)}{\phi(t)} \sum_{d  |t} \mu(\frac{t}{d}) d \sum_{\substack{n \leq N' \\ n \equiv a \pmod q \\ d | n}} 1 \\
    \end{align}
    by changing the order of summation and using that 
    \[ c_q(n) = \sum_{d | (q, n)} \mu(\frac{q}{d}) d.\]
    The inner-most sum is 
    \[ \frac{N'}{\text{lcm}(q, d)} \mathbf{1}_{\gcd(d,q) | a} + O(1),\] where the contribution from $O(1)$ \[ \sum_{t \leq Q} \frac{1}{\phi(t)} \sum_{d \mid t } d \lesssim \sum_{t \leq Q} \frac{\sigma_1(t)}{\phi(t)} \lesssim Q^{1+o(1)},   \] where $ \sigma_1(t)$ is the sum of the divisors of $t$. So the contribution of $O(1)$ is at most $O(Q^2)$, which can be absorbed into the error term. Therefore we have reduced \eqref{arithmetic-progression-sum-heath-brown} to the asymptotic:
    \begin{align} \label{arithmetic-progression-sum-common-part}
        \sum_{t \leq Q} \frac{\mu(t)}{\phi(t)} \sum_{d  | t} \mu(\frac{t}{d}) d \frac{N'}{\text{lcm}(q, d)} \mathbf{1}_{\gcd(q,d)  |a} = \frac{N'}{\phi(q)} \mathbf{1}_{(a, q) = 1} + O(N \exp(-c \log^{1/10} N))
    \end{align}
    for any $1 \leq N' \leq N$. Rearranging the least common multiple, we express the left-hand side of above equation as
    \begin{align} \label{arithmetic-progression-sum-common-part2}
        \sum_{t \leq Q} \frac{\mu(t)}{\phi(t)} \sum_{d  | t} \mu(\frac{t}{d}) d \frac{N'}{\text{lcm}(q, d)} \mathbf{1}_{\gcd(q,d) | a} &= \frac{N'}{q} \sum_{t \leq Q}  \frac{\mu(t)}{\phi(t)} \sum_{d  |t} \mu(\frac{t}{d}) \gcd(q, d)  \mathbf{1}_{\gcd(q,d)  |a}.
    \end{align}
    If $t$ does not divide $q$, then as $t$ is squarefree, we see that there must be a prime $r $ dividing $t$ but not $q$. This prime accounts for the vanishing of the inner sum: we can decompose the divisors of $t$ based on whether they are coprime with $r$. Explicitly, for each $d_0|t, \ (d_0,r) = 1$
    \[ \text{gcd}(q,d_0) = \text{gcd}(q,d_0r) \; \; \; \text{ while } \; \; \; \mu(\frac{t}{d_0}) = - \mu(\frac{t}{d_0 r}).\]
    Therefore $t | q$, so also $d | q$, so we may update our expression accordingly:
    \begin{align}
        \frac{N'}{q} \sum_{t \leq Q}  \frac{\mu(t)}{\phi(t)} \sum_{d  |t} \mu(\frac{t}{d}) \gcd(q, d)  \mathbf{1}_{\gcd(q,d)  |a} &= \frac{N'}{q} \sum_{\substack{t \leq Q \\ t | q}}  \frac{\mu(t)}{\phi(t)} \sum_{d  |(t, a)} \mu(\frac{t}{d}) d \\
        &= \frac{N'}{q} \sum_{\substack{t \leq Q \\ t | q}}  \frac{\mu(t)}{\phi(t)} \mu(\frac{t}{(t, a)}) \phi((t, a)),
    \end{align}
where above, we have abbreviated $(t,a) := \text{gcd}(t,a)$, and used the following identity coming from the multiplicative property of M\"{o}bius function and the fact that $t$ is square-free:
\begin{align*}
     \sum_{d  |(t, a)} \mu(\frac{t}{d}) d = \mu(t)\sum_{d  |(t, a)} \mu(d) d = \mu(t)\prod_{p\mid (t,a)} (1 - p ) = \mu(t) \mu((t,a)) \phi((t,a)).
\end{align*}
We now argue that we can remove the restriction, $t \leq Q$: for, if there exists $t > Q$ with $t | q$, then $q \geq Q$, which is impossible since $ q \leq Q^{\frac{1}{2}}.$
    If $p \mid (a,q)$ for some prime $p$, then we can obtain similar cancellation of this sum to \eqref{arithmetic-progression-sum-common-part2} by decomposing the divisors of $q$ based on whether the divisor is coprime to $p$. So we may assume $(a, q) = 1$, and in particular $(a, t) = 1$. Thus,
    \begin{align}
        \sum_{\substack{n \leq N' \\ n \equiv a \pmod q}} \Lambda_{\leq Q}(n) &= \frac{N'}{q} \mathbf{1}_{(a, q) = 1} \sum_{t | q}  \frac{\mu^2(t)}{\phi(t)} + O(Q^2).
    \end{align}
    To obtain \eqref{arithmetic-progression-sum-common-part}, we just apply the identity:
    \begin{align}
        \sum_{t | q} \frac{\mu(t)^2}{\phi(t)} = \frac{q}{\phi(q)}.
    \end{align}
    We now turn to the proof of \eqref{arithmetic-progression-sum-siegel}. 
    From the fundamental theorem of calculus we have:
    \begin{align}
        n^{\sigma - 1} \mathbf{1}_{[N']}(n) = \int_1^{N'} (1 - \sigma) M^{\sigma - 2} \mathbf{1}_{[M]}(n) dM + (N')^{\sigma - 1} \mathbf{1}_{[N']}(n)
    \end{align}
    and
    \begin{align}
        \frac{(N')^\sigma}{\sigma} - \frac{1}{\sigma} + 1 = \int_1^{N'} (1 - \sigma) M^{\sigma - 1} dM + (N')^{\sigma } ,
    \end{align}
    so from triangle inequality it suffices to show that:
    \begin{align} \label{arithmetic-progression-sum-siegel-2}
        \sum_{\substack{n \leq M \\ n \equiv a \pmod q}} \Lambda_{\leq Q}(n) \chi_{q_{\text{Siegel}}}(n) = \frac{M}{\phi(q)} \chi_{q_{\text{Siegel}}}(a) \mathbf{1}_{q_{\text{Siegel}}|q} \mathbf{1}_{(a, q) = 1} + O(N \exp(-c \log^{1/10}N))
    \end{align}
    for all $1 \leq M \leq N$. The left-hand side is given by
    \begin{align}   \label{e;InnerSum}
         \sum_{t \leq Q} \frac{\mu(t)}{\phi(t)} \sum_{d  |t} \mu(\frac{t}{d}) d \sum_{\substack{n \leq M \\ n \equiv a \pmod q \\ d | n}} \chi_{q_{\text{Siegel}}}(n). 
         \end{align}
Note that the sum is empty unless $(d,q)\mid a$.  And, otherwise, it is a sum over a progression of step size $\textup{lcm}(q,d)$. 
We have the formula below delivered to us from \cite[Formula 3.9]{iwaniec}. 
    \begin{align}
 \frac 1  {q_{\text{Siegel}} }
 \sum_{\substack{c \mod q_{\textup{Siegel}} }} \chi_{q_{\text{Siegel}}}(sc+a)
        = \begin{cases}
            \chi_{q_{\text{Siegel}}}(a) &   q_{\text{Siegel}}\mid s 
            \\
            0  &   q_{\text{Siegel}}\nmid s
        \end{cases}
    \end{align}
Apply this with $s= \textup{lcm}(q,d)$. It follows that 
    \begin{align}
        \sum_{\substack{n \leq M \\ n \equiv a \pmod q \\ d | n}} \chi_{q_{\text{Siegel}}}(n)
        = \frac{M}{\textup{lcm}(q,d)} \mathbf{1}_{(d, q)|a}
        \mathbf{1}_{q_{  \textup{Siegel} }\mid \textup{lcm}(q,d) }
        \chi_{q_{\text{Siegel}}}(a_d)+O(q_{\text{Siegel}}). 
    \end{align}
    where $n\equiv a_d\pmod{\textup{lcm}(q,d)}$, if and only if $n\equiv a\pmod q$ and $d|n$. Observe that the main term of the right-hand side vanishes when $q_{\text{Siegel}}$ does not divide $q$. Indeed,  $q_{  \textup{Siegel} }\mid \textup{lcm}(q,d) $, so  one finds a prime $r$ which divides both $d$ and $q_{\text{Siegel}}$. 
    So $n\equiv a_d\pmod{\textup{lcm}(q,d)}$ and  $r\mid n$ imply that $a_d\equiv 0\pmod{r}$. Hence, we infer that $r\mid (q_{\textup{Siegel}},a_d)$, which implies that $\chi_{q_{\textup{Siegel}}}(a_d)=0$.  So we may assume that  $q_{\text{Siegel}}\mid q$. Using this, the fact that $a_d\equiv n\equiv a\pmod{q}$, and periodicity of $\chi_{q_{\textup{Siegel}}}$, we learn that  $\chi_{q_{\textup{Siegel}}}(a_d)= \chi_{q_{\textup{Siegel}}}(a)$.  
    We conclude that the most inner sum in \eqref{e;InnerSum} is just 
    \[ \frac{M \chi_{q_{\text{Siegel}}}(a)}{\text{lcm}(q,d)} \mathbf{1}_{(d, q)|a} \mathbf{1}_{q_{\text{Siegel}} |q} + O(q_{\text{Siegel}}).\] The contribution from $O(q_{\text{Siegel}})$ to the sum \eqref{e;InnerSum} above is at most
    \begin{align}
        \lesssim \sum_{t \leq Q} \frac{1}{\phi(t)} \sum_{d  |t}  d q_{\textup{Siegel}}\lesssim Q^{1/2}\sum_{t \leq Q} \frac{\sigma_1(t)}{\phi(t)} \lesssim Q^{3/2 +o(1)} \lesssim Q^3,
    \end{align}
     which is bounded by 
    \[ O(N \exp(-c \log^{1/10}N)).\]
    Therefore we just need to verify that:
    \begin{align}
        \mathbf{1}_{q_{\text{Siegel}} |q} \chi_{\text{Siegel}}(a) \sum_{t \leq Q} \frac{\mu(t)}{\phi(t)} &\sum_{d  |t} \mu(\frac{t}{d}) d \frac{M}{\text{lcm}(q,d)} \mathbf{1}_{(d, q)|a} \\
        &= \frac{M}{\phi(q)} \chi_{q_{\text{Siegel}}}(a) \mathbf{1}_{q_{\text{Siegel}}|q} \mathbf{1}_{(a, q) = 1} + O(N \exp(-c \log^{1/10}N)),
    \end{align}
    which follows directly from \eqref{arithmetic-progression-sum-common-part}.
\end{proof}

\end{document}